\documentclass[reqno]{amsart}

\usepackage{versions}
\usepackage{xfrac}
\usepackage{ulem} 
\usepackage[utf8]{inputenc} 
\usepackage{amssymb}
\usepackage{mathrsfs}
\usepackage{graphicx}
\usepackage{latexsym}
\usepackage{stmaryrd}
\usepackage{enumitem}
\usepackage[bookmarks,hyperfootnotes=false, psdextra=true,hypertexnames=false]{hyperref}
\usepackage{multirow}
\usepackage{xcolor}
\usepackage{caption}
\colorlet{darkishRed}{red!80!black}
\colorlet{darkishBlue}{blue!60!black}
\colorlet{darkishGreen}{green!60!black}
\hypersetup{
    draft = false,
    bookmarksopen=true,
    colorlinks,
    linkcolor={red!60!black},
    citecolor={green!60!black},
    urlcolor={blue!60!black}
}
\usepackage[nameinlink, capitalise, noabbrev]{cleveref}
\usepackage{thmtools}
\usepackage{nameref}
\crefformat{enumi}{#2#1#3}
\crefformat{equation}{#2(#1)#3}
\usepackage[msc-links,abbrev]{amsrefs} 
\usepackage{doi}

\renewcommand{\PrintDOI}[1]{\doi{#1}}

\let\setminus=\smallsetminus
\usepackage{tikz}
\usepackage{tikz-cd}
\usetikzlibrary{calc,through,intersections,arrows, trees, positioning, decorations.pathmorphing, cd, patterns}
\usepackage{relsize}
\usepackage{comment}
\usepackage{svg}
\usepackage{mathtools}
\usepackage{nccmath}
\usepackage{pifont}

\usepackage{geometry}
\let\setminus=\smallsetminus
\renewcommand{\leq}{\leqslant}
\renewcommand{\geq}{\geqslant}
\renewcommand{\ge}{\geq}
\renewcommand{\le}{\leq}

\let\rho=\varrho
\let\phi=\varphi

\newcommand{ \N } { \mathbb{N} }

\makeatletter

\def\calCommandfactory#1{%
  \expandafter\def\csname c#1\endcsname{\mathcal{#1}}}
\def\frakCommandfactory#1{%
  \expandafter\def\csname frak#1\endcsname{\mathfrak{#1}}}
   
\newcounter{ctr}
\loop
  \stepcounter{ctr}
  \edef\X{\@Alph\c@ctr}
  \expandafter\calCommandfactory\X
  \expandafter\frakCommandfactory\X
  \edef\Y{\@alph\c@ctr}
  \expandafter\frakCommandfactory\Y
\ifnum\thectr<26
\repeat

\renewcommand{\cC}{\mathscr{C}}

\setenumerate{label={\normalfont (\roman*)}}

\usepackage[presets={vec-cev,abc,ABC,cAcBcC}]{letterswitharrows}

\newtheorem{theorem}{Theorem}[section] 
\newtheorem{proposition}[theorem]{Proposition}
\newtheorem{corollary}[theorem]{Corollary}
\newtheorem{lemma}[theorem]{Lemma}

\newtheorem{conjecture}[theorem]{Conjecture}

\newtheorem{mainresult}{Theorem} 
\crefname{mainresult}{Theorem}{Theorems}

\newtheorem{innercustomthm}{}

\newenvironment{customthm}[1]
  {\renewcommand\theinnercustomthm{#1}\innercustomthm}
  {\endinnercustomthm}

\theoremstyle{definition}
\newtheorem{example}[theorem]{Example}

\newtheorem{definition}[theorem]{Definition}

\newtheorem{question}[theorem]{Question}

\theoremstyle{remark}

\newtheorem{claim}{Claim}
\crefname{claim}{Claim}{Claims}
\usepackage{etoolbox}
\AtEndEnvironment{proof}{\setcounter{claim}{0}}

\newenvironment{claimproof}{\noindent\textit{Proof.}}{\strut\hfill\ensuremath{\blacksquare}\medskip}

\newcommand{\radialWidth}{radial width}
\newcommand*{\radialHWidth}[1]{radial~{#1}-width}

\DeclareMathOperator{\rw}{radw}

\DeclareMathOperator{\rHw}{radw_{\cH}}
\newcommand{\radwDec}[1]{\rw(#1)}

\DeclareMathOperator{\rs}{rads}

\DeclareMathOperator{\rad}{rad}

\newcommand{\dist}{d}

\newcommand{\CH}{\ensuremath{\mathcal H}}

\newcommand{\CV}{\ensuremath{\mathcal V}}
\newcommand{\HdecV}{\ensuremath{(H,\CV)}}

\newcommand{\subdivk}[2]{\ensuremath{T_{#2}#1}}
\newcommand*{\ball}[3]{\ensuremath{B_{#1}(#2,#3)}}

\newcommand{\gd}{graph-decom\-posi\-tion}

\newcommand{\tds}{tree-decom\-posi\-tions}

\newcommand{\COMMENT}[1]{}

\newbool{arXiv}
\booltrue{arXiv} 

\newcommand{\arXivOrNot}[2]{\ifbool{arXiv}{{#1}}{{#2}}}

\tikzset{vtx/.style={fill, circle, inner sep = 1.2pt}}
\tikzset{branch/.style={fill, rectangle, inner sep = 2pt}}

\title[A structural duality for path-decompositions into parts of small radius]{A structural duality for path-decompositions\\ into parts of small radius}
\author[S.\ Albrechtsen \and R.\ Diestel \and A.\ Elm \and E.\ Fluck \and R.\ Jacobs \and P.\ Knappe \and P.\ Wollan]{Sandra Albrechtsen \and Reinhard Diestel \and Ann-Kathrin Elm \and Eva Fluck \and Raphael W. Jacobs \and Paul Knappe \and Paul Wollan}

\address{University of Hamburg, Department of Mathematics, Bundesstraße 55 (Geomatikum), 20146 Hamburg, Germany}
\email{sandra@albrechtsen-mail.de, \{raphael.jacobs, paul.knappe\}@uni-hamburg.de}

\address{Universität Heidelberg, Department of Computer Science, Im Neuenheimer Feld 205, 69120 Heidelberg, Germany}
\email{ann-kathrin.elm@web.de}

\address{RWTH Aachen University, Department of Computer Science, Ahornstraße 55, 52074 Aachen, Germany}
\email{fluck@informatik.rwth-aachen.de}

\address{University of Rome, “La Sapienza”, Department of Computer Science, Via Salaria 113, 00198 Rome, Italy}
\email{wollan@di.uniroma1.it}

\date{}

\keywords{Path-Decompositions, Graph-Decompositions, Radial Width, Quasi-Isometry, Geodesic, Topological Minors, Fat minors}
\@namedef{subjclassname@2020}{\textup{2020} Mathematics Subject Classification}
\subjclass[2020]{05C10 (Primary) 05C75, 05C12, 05C62, 05C83 (Secondary)}

\begin{document}

\begin{abstract}
    It is an easy observation that if a graph~$G$ admits a path-decomposition whose parts have small radius, then $G$ contains no large subdivision of $K_{1,3}$ or $K^3$ as a (quasi-)geodesic subgraph.
    We show that these are in fact the only obstructions to such path-decompositions of small radial width, and we prove analogous results for decompositions modelled on cycles and subdivided stars instead of paths.
    
    With our results we confirm in a strong form a conjecture of Georgakopoulos and Papasoglu on fat-minor-characterisations of graphs quasi-isometric to paths, cycles and paths, and subdivided stars, respectively.
    For this, we present a novel view on quasi-isometries between graphs by graph-decompositions of bounded radial width and spread.   
    This new perspective enables us to prove further results in coarse graph theory, and may thus be of independent interest.
\end{abstract}

\maketitle

\section{Introduction\texorpdfstring{\protect\footnote{This paper was first published on arXiv on 17.07.2023. This version has a completely rewritten Introduction, which takes into account several developments occurring close or parallel to the first publication; after the Introduction, we made only restructurings, and added a few propositions and some small corrections and amendments.}}{}}

\subsection{Fat minors and quasi-isometries}

Following Gromov's ideas on coarse geometry~\cite{GromovInfiniteGroups} into the realm of graphs, Georgakopoulos and Papasoglu~\cite{georgakopoulos2023graph} suggested a study of graphs from a coarse or metric perspective, which revolves around the concept of \emph{quasi-isometry}.
Roughly speaking, two metric spaces are quasi-isometric if their large-scale geometry coincides, and more formally, a quasi-isometry is a generalisation of bi-Lipschitz maps that allows for an additive error.
For example, all locally finite Cayley graphs of a finitely generated group are quasi-isometric to each other (see e.g.~\cite{loh2017geometric}*{Proposition 5.2.5}).

As their favorite problem of metric graph-theoretic flavour, Georgakopoulos and Papasoglu proposed the following conjecture, whose qualitative converse is immediate:
\begin{conjecture}[\cite{georgakopoulos2023graph}*{Conjecture 1.1}] \label{conj:GeorgakopoulosPapasoglu}
     Let~$\cX$ be a finite set of finite graphs. Then there exists a function $f : \N \rightarrow \N \times \N$ such that, for all integers $K$, every graph with no $K$-fat $X$ minor for any $X \in \cX$ is $f(K)$-quasi-isometric to a graph with no $X$ minor for any $X \in \cX$.
\end{conjecture}
\noindent
Here, `$K$-fat minors' are a metric variant of minors:
Roughly speaking, a $K$-fat $X$ minor is an $X$ minor with additional distance constraints: its branch sets and branch paths are pairwise at least~$K$ apart, except for incident vertex-edge pairs (see~\cref{subsec:GraphDecQuasiGeoTopMinors} for the formal definition).

Georgakopoulos and Papasoglu~\cite{georgakopoulos2023graph} verified their conjecture for~$\cX = \{K^3\}$ (which describes graphs that are quasi-isometric to a forest), as well as for~$\cX = \{K_{1,m}\}$.
An earlier result of Chepoi, Dragan, Newman, Rabinovich and Vaxes~\cite{FatK23Minor} yields the case~$\cX = \{K_{2,3}\}$ (quasi-isometric to an outerplanar graph).
Fujiwara and Papasoglu~\cite{coarsecacti} solved the~$\cX = \{K_4^-\}$-case (quasi-isometric to a cactus); see also~\cite{FatK4} for a simpler proof. Moreover, Albrechtsen, Jacobs, Knappe and Wollan~\cite{FatK4} solved the case $\cX = \{K^4\}$.
In contrast to these positive results, Davies, Hickingbotham, Illingworth and McCarty~\cite{CounterexAgelosPanosConjecture} showed that~\cref{conj:GeorgakopoulosPapasoglu} is false in general.

In this paper, we establish~\cref{conj:GeorgakopoulosPapasoglu} for three further small cases but in a stronger form:
\begin{mainresult} \label{main:TheoremIntro}
    \cref{conj:GeorgakopoulosPapasoglu} is true for~$\cX = \{K^3, K_{1,3}\}$ (quasi-isometric to a disjoint union of paths), $\cX = \{K_{1,3}\}$ (cycles and paths)\footnote{Georgakopoulos and Papasoglu prove the more general case~$\cX = \{K_{1,m}\}$ in a second version of their paper~\cite{georgakopoulos2023graph}, which appeared after the initial publication of the present paper.} and~$\cX = \{K^3, W\}$ (subdivided stars), where~$W$ is the graph depicted in~\cref{fig:wrench}.
    In all these cases, \cref{conj:GeorgakopoulosPapasoglu} even holds true when replacing `fat minors' by `quasi-geodesic topological minors'.
\end{mainresult}
\noindent We refer the reader to \cref{thm:path-width-intro,thm:star-width-intro,thm:cycle-width-intro} for the respective functions $f$.

In~\cref{main:TheoremIntro}, `quasi-geodesic topological minors' are a metric version of topological minors, whose model in~$G$ is `quasi-geodesically' embedded (see~\cref{subsec:GraphDecQuasiGeoTopMinors} for the definition).
Roughly speaking, the metric of the topological minor agrees (up to a multiplicative constant) with the metric induced by the graph $G$.
In particular, quasi-geodesic topological minors yield fat minors but in general not the other way around (\cref{lem:QuasiGeodTopMinorAreFatMinors}).
\medskip

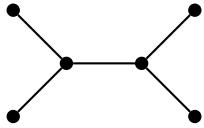
\begin{figure}
	\begin{tikzpicture}[thick, every node/.style={draw, circle, fill=black, inner sep=1.5pt}]
		
		\path (-1.85,0.1) node (A) {};
		\path (-1.5,-0.5) node (B) {};
		\path (-1,0) node (C) {};
		\path (1,1) node (D) {};
		\path (1.5,1.5) node (E) {};
		\path (1.85,0.9) node (F) {};
		
		\draw (A) -- (C) -- (D) -- (E) (B) -- (C) (D) -- (F);
		
	\end{tikzpicture}
	\caption{The {\em wrench\/} graph $W$.} \label{fig:wrench}
\end{figure}

Our work on~\cref{main:TheoremIntro} was independent of~\cref{conj:GeorgakopoulosPapasoglu} and the involved concepts such as fat minors and quasi-isometries.
In fact, we only discovered the connections of our results to quasi-isometries and fat minors, and in particular to~\cref{conj:GeorgakopoulosPapasoglu}, through~\cite{georgakopoulos2023graph}.
Our approach is not based on quasi-isometries, but on the notion of graph-decompositions, which we describe in what follows.

\subsection{Graph-decompositions} \label{subsec:Intro:GraphDecs}

Graph-decompositions~\cite{GraphDec} are a natural extension of \tds\ which allow the bags $V_h$ of decompositions $(H, \cV)$ to be arranged along general decomposition graphs $H$ instead of just trees.
We then say that $(H, \cV)$ is \emph{modelled on $H$} and call it an \emph{$H$-decomposition}.
Moreover, given any graph class $\cH$, we refer to $H$-decompositions with $H \in \cH$ as \emph{$\cH$-decompositions}.
Recent applications of graph-decompositions include a local-global decomposition theorem~\cite{GraphDec} as well as the study of local separations~\cite{LocalSeps} and of locally chordal graphs~\cite{LocallyChordal}.

In this paper, we present a width-notion for graph-decompositions such that a graph~$G$ has an $H$-decomposition of small width if and only if $H$ resembles the large-scale structure of~$G$.
The naive approach defines the `width' of a graph-decomposition analogously to tree-width, that is, as the minimal cardinality of a bag of the decomposition (minus $1$).
However, the respective~`$\cH$-width', the minimal width of an $\cH$-decomposition for a given graph class~$\cH$, does not yield a meaningful extension of tree-width: if a minor-closed class $\cH$ of graphs has bounded tree-width, then every graph of small $\cH$-width has small tree-width itself, and if $\cH$ has unbounded tree-width, then the $\cH$-width of every graph is at most $2$ (minus $1$) \cites{GMhierarchies,GMV}.\footnote{By Diestel and Kühn~\cite{GMhierarchies}*{Proposition~3.7}, every graph $G$ has a grid-decomposition of width at most $2$ (minus~$1$). Thus, the $\cH$-width of $G$ is at most $2$ (minus $1$) for any graph class $\cH$ of unbounded tree-width by Robertson and Seymour's Grid Theorem~\cite{GMV}.}

This inspired us to consider a metric perspective instead: To define the `width' of a graph-decomposition, we measure the size of its bags not in terms of their cardinality, but by the radius of its \emph{parts}, the induced subgraphs on the bags of the decomposition.
More formally, recall that the \emph{radius} of a graph~$G$ is the smallest~$r \in \N$ such that some vertex of~$G$ has distance at most~$r$ to all vertices of~$G$.
We then let the \emph{(inner-)\radialWidth} of a decomposition be the largest radius among its parts and define the~\emph{\radialHWidth{$\cH$}} of~$G$ for a given class~$\cH$ of graphs to be the smallest {\radialWidth} among all decompositions of~$G$ modelled on graphs in~$\cH$.

This notion of {\radialWidth} is not new for tree-decompositions.
Indeed, the \radialHWidth{$\cH$} for the class~$\cH$ of all trees, or \emph{\radialHWidth{tree}} for short, has been studied before, e.g. as the equivalent \emph{tree-length} in~\cite{dourisboure2007treelength} and \emph{tree-breadth} in~\cite{dragan2014approximation}.
Similar to the classical tree-width notion~\cite{GMII}, several computationally hard problems such as the computation of the metric dimension of a graph~\cite{belmonte2017metric} can be efficiently solved on graphs of small \radialHWidth{tree} (for a summary, see~\cite{dourisboure2007treelength}*{Section~1} or~\cite{leitert2017treebreadth}).

\subsection{Interplay between graph-decompositions and quasi-isometries}

As it turns out, \gd s and quasi-isometries are closely related. In fact, these notions become qualitatively equivalent if we restrict to `honest' graph-decompositions of bounded radial width and bounded `radial spread'.

An $H$-decomposition of~$G$ is~\emph{honest} if all its bags are non-empty and for every edge of its decomposition graph~$H$, the bags corresponding to its endnodes intersect.
Informally speaking, being honest ensures that all connectivity in the decomposition graph~$H$ also appears in the decomposed graph~$G$.
For each vertex~$v \in G$ let~$H_v$ be the induced subgraph of~$H$ on the set of all nodes~$h \in H$ whose corresponding bags contain~$v$.
The \emph{(inner-)radial spread} of the $H$-decomposition is then defined as the largest radius of the~$H_v$ with~$v \in V(G)$.

The equivalence between the existence of an honest~$H$-decomposition with bounded radial width and spread and of a bounded quasi-isometry to~$H$ was observed for the case of trees~$H$ by Berger and Seymour~\cite{bergerseymourboundeddiamTD}*{4.1}.
Here, we extend their observation to arbitrary graphs~$H$.

\begin{proposition} \label{prop:QuasiIsoIntro}
    There exist functions $g: \N^2 \to \N^2$ and $h_1, h_2: \N^2 \to \N$ such that the following holds for all graphs $G,H$:
    \begin{enumerate}
        \item \label{itm:QuasiIsoIntro:GD-Iso} If~$G$ admits an honest~$H$-decomposition of radial width~$r_0$ and radial spread~$r_1$, then $G$ is~$g(r_0,r_1)$-quasi-isometric to~$H$.
        \item \label{itm:QuasiIsoIntro:Iso-GD} If~$G$ is~$(L,C)$-quasi-isometric to~$H$, then~$G$ admits an honest~$H$-decomposition of radial width~$h_1(L,C)$ and radial spread~$h_2(L,C)$.
    \end{enumerate}

\end{proposition}
\noindent
We refer the reader to~\cref{sec:QuasiIso} for the detailed statement on the functions~$g$ and~$h_1, h_2$.

The correspondence given by \cref{prop:QuasiIsoIntro} paves the way for a new proof method towards~\cref{conj:GeorgakopoulosPapasoglu}:
one through the graph-theoretic construction of a suitable graph-decomposition.
We follow this approach in the present paper and further demonstrate its power and versatility in~\cite{FatK4}.

\cref{prop:QuasiIsoIntro} in particular implies that we reached our initial goal from the beginning of~\cref{subsec:Intro:GraphDecs} with the notions of `radial width' and `radial spread':
a graph $H$ resembles the large-scale geometry of~$G$ (in terms of bounded quasi-isometries) if and only if $G$ admits an honest $H$-decomposition of bounded radial width and spread.

\subsection{Our results} \label{subsec:GraphDecQuasiGeoTopMinors}
With a suitable width measure at hand, our next goal was to identify obstructions to small radial $\cH$-width and to characterise which graphs have small radial $\cH$-width for given classes $\cH$ of graphs. Similar to \cref{conj:GeorgakopoulosPapasoglu}, we considered metric versions of minors as candidates for suitable obstructions. For the graph classes~$\cH$ that we study in this paper, our obstructions are `quasi-geodesic topological minors', which can be seen as a special case of fat minors. Let us make this precise in what follows.

A \emph{$(\geq k)$-subdivision} of a graph~$X$ is a graph~$X'$ which arises from~$X$ by subdividing each edge at least $k$ times, i.e.\ replacing every edge with a path of at least~$k+1$ edges.
Further, a subgraph~$X$ of a graph~$G$ is \emph{$c$-quasi-geodesic}\footnote{Note that in metric spaces this property is often called $(c,0)$-quasi-geodesic. In~\cite{BendeleRautenbach22} it is called $c$-multiplicative.} for some~$c \in \N$ if the distance of any two vertices~$x$ and~$y$ in~$X$ is at most~$c$ times their distance in~$G$.
Here, the parameter~$c$ describes how well the distances in~$X$ approximate the distances in~$G$; in particular, a subgraph is geodesic if and only if it is~$1$-quasi-geodesic.

Quasi-geodesic topological minors are indeed an obstruction to small radial width: 
if a $(\geq k)$-subdivision of a graph $X$ is a $c$-quasi-geodesic subgraph of $G$, then $X$ is a $\lfloor \frac{k-2}{2c} \rfloor$-fat minor of $G$ (\cref{lem:QuasiGeodTopMinorAreFatMinors}).
Hence, as fat minors form an obstruction to small radial width (\cref{lem:FatMinorObstructRadialWidth}), so do quasi-geodesic topological minors (\cref{lem:quasigeodesictopologicalminorsareobstructions}).

Our main result, \cref{main:TheoremIntro}, asserts that these are in fact the only obstructions to small radial $\cH$-width for the three cases where~$\cH$ consists of paths, cycles and paths, and subdivided stars, respectively.
We give the detailed statements in the following~\cref{thm:path-width-intro,thm:cycle-width-intro,thm:star-width-intro}, from which we then deduce~\cref{main:TheoremIntro}.

\begin{mainresult}[Radial path-width]\label{thm:path-width-intro}
	Let~$k \in \N$.
    If a connected graph~$G$ contains no~$(\geq k)$-subdivision of~$K^3$ as a geodesic subgraph and no~$(\geq 3k)$-subdivision of~$K_{1,3}$ as a~$3$-quasi-geodesic subgraph, then $G$ admits an honest decomposition modelled on a path $P$ of radial width at most~$18k + 2$ and radial spread at most~$18k+1$.

    Moreover, $P$ is $(1, 18k+2)$-quasi-isometric to $G$.
\end{mainresult}

\begin{mainresult}[Radial cycle-width]\label{thm:cycle-width-intro}
	Let~$k \in \N$.
    If a connected graph~$G$ contains no~$(\geq 3k)$-subdivision of~$K_{1,3}$ as a~$3$-quasi-geodesic subgraph, then $G$ admits an honest decomposition modelled on a cycle or path~$C$ of radial width at most~$18k+2$ and radial spread at most~$36k+2$.

    Moreover, $C$ is $(1, 18k+2)$-quasi-isometric to $G$. 
\end{mainresult}

\begin{mainresult}[Radial star-width]\label{thm:star-width-intro}
	Let~$k \in \N$.
    If a connected graph~$G$ contains no~$(\geq k)$-subdivision of~$K^3$ as a geodesic subgraph and no~$(\geq 3k)$-subdivision of the wrench graph~$W$ as a~$3$-quasi-geodesic subgraph, then $G$ admits an honest decomposition modelled on a subdivided star $S$ of radial width at most~$58k + 9$ and radial spread at most $30k+7$. 

    Moreover, there exists some $C_k \in \N$ such that some subdivided star is $(1, C_k)$-quasi-isometric to~$G$.\footnote{While we will only formally prove that some subdivided star $S'$ and some $C_k \in \N$ exists, one can check by carefully reading the proof that we may choose $S' = S$ and $C_k = 60k+14$ (see the paragraph after the proof of \cref{thm:star-width-intro} in \cref{sec:StarWidth} for details).} 
\end{mainresult}

\begin{proof}[Proof of \cref{main:TheoremIntro} given \cref{thm:path-width-intro,thm:cycle-width-intro,thm:star-width-intro}]
    Let $K \in \N$, and let $G$ be a graph with no $K$-fat $X$ minor for any $X$ in the respective $\cX$.
    By \cref{lem:QuasiGeodTopMinorAreFatMinors} there is an integer $k$ depending on $K$ and $\cX$ only, such that,
    for every $c \in \N$, $G$ contains no $(\geq c \cdot k)$-subdivision of any $X \in \cX$ as $c$-quasi-geodesic subgraph.
    Now we deduce \cref{main:TheoremIntro} from \cref{thm:path-width-intro,thm:cycle-width-intro,thm:star-width-intro} by applying the respective theorem to the components of $G$ and combine the obtained $(1,C_k)$-quasi-isometries to one from the disjoint union $H$ of their domains to $G$.
    \cref{lem:InverseQI} yields the desired $f(K)$-quasi-isometry from $G$ to $H$.
\end{proof}

\subsection{Sketch of the proofs}\label{sec:sketchproofs}

Let us give a brief overview of the proof techniques for~\cref{thm:path-width-intro,thm:star-width-intro,thm:cycle-width-intro}.
For the proof of~\cref{thm:path-width-intro} (radial path-width), we start with a longest geodesic path~$P$ in the connected graph~$G$.
We then show that either balls of small radius around~$V(P)$ cover all of~$G$ or we can find a geodesic \hbox{$(\geq k)$-subdivision} of the triangle~$K^3$ or a~$3$-quasi-geodesic~$(\geq 3k)$-subdivision of the claw~$K_{1,3}$.
In the former case we construct a~$P$-decomposition of~$G$ by letting the bag corresponding to a node~$p \in P$ be the union of all those small radius-balls, whose centre vertex has small distance to~$p$ in~$P$.

The proof technique for \cref{thm:path-width-intro} immediately generalises to \cref{thm:cycle-width-intro} (radial cycle-width).
The proof of \cref{thm:star-width-intro} (radial star-width) is more involved. 
As before we start with a longest geodesic path~$P$ in~$G$ and consider the subgraph~$G'$ of~$G$ which is covered by balls of small radius around~$V(P)$.
Unlike in the path-case, $G'$ will in general not be equal to~$G$.
But if $G$ does not contain a large subdivision of~$K^3$ or~$W$ as a~($3$-quasi-)geodesic subgraph, then every component of $G-G'$ will have its neighbours in $G'$ only close to the same vertex $p_s$ of $P$.
We then identify a geodesic path within each component of~$G - G'$ such that all vertices in the component have small distance to this path.
All these paths can then be combined with~$P$ into a subdivided star~$S$, by adding edges between their last vertices and $p$. 
We then assign to each node $s$ of $S$ a suitable bag $V_s$ of vertices of bounded distance to $s$.

\subsection{An open conjecture about quasi-isometries to trees}

For the special case of~$\cX = \{K^3\}$, Georgakopoulos and Papasoglu answered~\cref{conj:GeorgakopoulosPapasoglu} in the affirmative, proving that the absence of a~$K$-fat~$K^3$ minor implies the existence of an~$f(K)$-quasi-isometry to a forest.
Berger and Seymour~\cite{bergerseymourboundeddiamTD} characterised the graphs quasi-isometric to a forest using a similar kind of obstruction. 
Our results~\cref{thm:path-width-intro,thm:cycle-width-intro,thm:star-width-intro} yield another natural guess for such an obstruction -- long quasi-geodesic cycles:

\begin{conjecture}\label{conj:tree-width-intro}
	There is a function~$f:\mathbb{N} \rightarrow \mathbb{N} \times \mathbb{N}$ such that if a connected graph~$G$ does not contain a~$c$-quasi-geodesic cycle of length at least $3ck$ for some~$c, k \in \N$, then~$G$ is $f(k)$-quasi-isometric to a tree.
\end{conjecture}
\noindent If true, this statement would strengthen the respective results of Georgakopoulos and Papasoglu and of Berger and Seymour.
Note that by~\cref{prop:QuasiIsoIntro} and \cite{bergerseymourboundeddiamTD}*{4.1} we can equivalently ask about the existence of a function~$f$ which guarantees that $G$ admits an honest tree-decomposition of radial width at most~$f(k)$.
\arXivOrNot{In \cref{app:Outlook} we show that the proof technique described in \cref{sec:sketchproofs} for \cref{thm:path-width-intro,thm:cycle-width-intro} does not work for \cref{conj:tree-width-intro}.}{}

\arXivOrNot{In \cref{app:CterexQuestionRadWidth}, we show that the strengthening of~\cref{conj:GeorgakopoulosPapasoglu} for quasi-geodesic topological minors analogous to \cref{conj:tree-width-intro} fails  earlier than~\cref{conj:GeorgakopoulosPapasoglu} itself.}{}

\subsection{Structure of the paper}

\cref{sec:Preliminaries} collects some basic definitions.
In \cref{sec:RadialWidth} we introduce the notion of graph-decompositions and their radial width and spread, and show how they are related to quasi-isometries, proving \cref{prop:QuasiIsoIntro}. We also prove in \cref{subsec:QuasiGeoTopMinors} that quasi-geodesic topological minors yield fat minors and are an obstruction to small radial width.
The following three~\cref{sec:PathWidth,sec:CycleWidth,sec:StarWidth} contain the proofs of \cref{thm:path-width-intro,thm:cycle-width-intro,thm:star-width-intro}, respectively.
\arXivOrNot{We finally discuss further problems resulting from our work in~\cref{app:Outlook,app:CterexQuestionRadWidth}.}{In the appendix of the arXiv version we discuss further problems resulting from our work. In particular, we show that the proof technique described in \cref{sec:sketchproofs} for \cref{thm:path-width-intro,thm:cycle-width-intro} does not work for \cref{conj:tree-width-intro}, and we show that the strengthening of~\cref{conj:GeorgakopoulosPapasoglu} for quasi-geodesic topological minors fails  earlier than~\cref{conj:GeorgakopoulosPapasoglu} itself.}

\section{Preliminaries} \label{sec:Preliminaries}

We mainly follow the notations from~\cite{DiestelBook}. We recall that for two sets $X,Y$ of vertices of a graph $G$ an \emph{$X$--$Y$ path} in $G$ is a path in $G$ whose only vertex in $X$ is its first vertex and whose only vertex in $Y$ is its last vertex. Moreover, for a subgraph $H$ of $G$ an \emph{$H$-path} in $G$ is a non-trivial path in $G$ which meets $H$ precisely in its endvertices. For the remainder of the section, we briefly present the definitions we will require going forward.
All graphs in this paper are finite.

Let $G$ be a graph.
We write~$\dist_G(v, u)$ for the distance of the two vertices~$v$ and~$u$ in~$G$. 
For two sets~$U$ and~$U'$ of vertices of~$G$, we write~$\dist_G(U, U')$ for the minimum distance of two elements of~$U$ and~$U'$, respectively.
If one of~$U$ or~$U'$ is just a singleton, then we may omit the braces.
Further, if~$X$ is a subgraph of~$G$, then we abbreviate~$d_G(U, V(X))$ as~$d_G(U, X)$ for notational  simplicity.

Given a set~$X$ of vertices of~$G$, the \emph{$r$-ball (in~$G$) around~$X$}, denoted as~$B_G(X, r)$, is the set of all vertices in~$G$ of distance at most~$r$ from~$X$ in~$G$.
If~$X = \{v\}$ for some~$v \in V(G)$, then we omit the braces, writing~$B_G(v, r)$ for the~$r$-ball (in $G$) around~$v$.

Further, the \emph{radius}~$\rad(G)$ of~$G$ is the smallest number~$k \in \N$ such that there exists some vertex~$v \in G$ with~$\dist_G(v, w) \le k$ for every vertex~$w \in G$.
Note that~$G$ has radius at most~$k$ if and only if there is some vertex~$v \in G$ with~$V(G) = B_G(v, k)$.
If~$G$ is not connected, then we define its radius to be~$\infty$.
Additionally, if $U \subseteq V(G)$, then the \emph{radius of $U$ in $G$} is the smallest number $k \in \N$ such that there exists some vertex $v \in G$ with $U \subseteq B_G(v, k)$.

\section{Graph-decompositions and quasi-isometries} \label{sec:RadialWidth}

In this section we first introduce the concept of graph-decompositions and their radial width and spread and collect some properties of these notions.
We then study the connection between graph-decompositions of small radial width and spread and quasi-isometries to their decomposition graph, proving \cref{prop:QuasiIsoIntro}. Finally, we define quasi-geodesic topological minors and show that they yield fat minors, which indeed form obstruction to small radial width.

\subsection{Graph-decompositions}\label{subsec:GDs}

Let us recall the notion of graph-decompositions:
\begin{definition}[Graph-decomposition \cite{GraphDec}] \label{def:graph_decomp}
	Let~$G$ and~$H$ be graphs and let~$\cV=(V_h)_{h\in H}$ be a family of sets~$V_h$ of vertices of~$G$.\footnote{In~\cite{GraphDec} a graph-decomposition is, more generally, defined as a pair~$(H, (G_h)_{h \in H})$ of a graph~$H$ and a family of (not necessarily induced) subgraphs~$G_h$ of~$G$ indexed by the nodes of~$H$, instead of a family of vertex sets~$V_h$ as in this paper.}
	We call \HdecV\ an \emph{$H$-decomposition} of~$G$, a \emph{decomposition of~$G$ modelled on~$H$}, or just a \emph{graph-decomposition}, if
	\begin{enumerate}[label=(H\arabic*)]
		\item $\bigcup_{h\in H} G[V_h] = G$, and \label{ax:GraphDec1}
		\item for every vertex~$v \in G$, the graph of~$H_v := H[\{h\in H\mid v\in V_h\}]$ is connected. \label{ax:GraphDec2}
	\end{enumerate}
	The sets~$V_h$ are called the \emph{bags} of this graph-decomposition, their induced subgraphs~$G[V_h]$ are its \emph{parts}, and the graph~$H$ is its~\emph{decomposition graph}.
    Whenever a graph-decomposition is introduced as~$(H, \cV)$, we tacitly assume~$\cV = (V_h)_{h \in H}$.
    For a class~$\cH$ of graphs, we call \HdecV\ an \emph{$\cH$-decomposition} if it is an $H$-decomposition with~$H \in \cH$. \end{definition}

A graph-decomposition $(H,(V_h)_{h \in H})$ of a graph $G$ is \emph{honest} if all its bags $V_h$ are non-empty and $V_h \cap V_{h'} \neq \emptyset$ for every edge $hh' \in H$.

If $G$ is connected and every bag of $(H, \cV)$ is non-empty, then it follows from~\cref{def:graph_decomp} that~$H$ has to be connected as well. 
More generally, we have the following lemma:

\begin{lemma} \label{lem:GraphDecConnSubgraph}
    Let~$(H, \cV)$ be a graph decomposition of a graph~$G$.
    Then for any connected subgraph~$X$ of~$G$, the graph~$H_X := H[\{h \in H \mid X \cap V_h \neq \emptyset\}]$ is connected. 
\end{lemma}
\begin{proof}
    \cref{ax:GraphDec1} yields for every edge $vw$ in $H$ a part $G[V_h]$ that contains $vw$.
    This implies that both~$H_v$ and~$H_w$ contain $h$ and hence intersect. 
    Since $H_v$ and $H_w$ are connected by \cref{ax:GraphDec2}, it follows that $H_v \cup H_w$ is connected as well.
    This implies that as $X$ is connected, $H_X = \bigcup_{x \in X} H_x$ is connected as well. 
\end{proof}

The next lemma shows that we can obtain a new graph-decomposition by enlarging all bags of a given graph-decomposition simultaneously in that we replace each of its bags by the~$r$-ball around it for some globally fixed~$r \in \N$.
A version of~\cref{lem:BallsAroundGraphDec} for tree-decompositions was first proven in \cite{GraphDecArxiv}*{Lemma A.1}.

\begin{lemma} \label{lem:BallsAroundGraphDec}
    Let~$(H, \cV)$ be a graph-decomposition of a graph~$G$, and let~$r \in \N$.
    For every vertex~$h \in H$, we let~$V'_h := B_G(V_h, r)$.
    Then~$(H, \cV')$ is again a graph-decomposition of~$G$.
\end{lemma}

\begin{proof}
    For notational simplicity, we adopt the following conventions for the course of this proof.
    Given a subgraph~$X$ of~$G$, we write~$H_X := H[\{h \in H\mid V_h \cap V(X) \neq \emptyset\}]$ and~$H'_X := H[\{h \in H \mid V'_h \cap V(X) \neq \emptyset\}]$.
    
    By definition,~$(H, \cV')$ satisfies~\cref{ax:GraphDec1}.
    For~\cref{ax:GraphDec2} consider any vertex~$v \in G$.
    Then~$H_v$ is connected by~\cref{ax:GraphDec2} and non-empty by~\cref{ax:GraphDec1}.
    So in order to prove that~$H'_v$ is connected as well, it suffices to show that for every~$h \in V(H'_v)$, there is an~$h$--$V(H_v)$~path in~$H'_v$.

    By the definition of~$V'_h$, there exists~$w \in V_h$ with~$\dist_G(v, w) \le r$, and we fix a shortest~$v$--$w$~path~$P$ in~$G$.
    Then every vertex~$p \in P$ satisfies~$\dist_G(v, p) \leq r$ as witnessed by~$P$.
    In particular, every node~$h'  \in H$ with~$V_{h'} \cap V(P) \neq \emptyset$ satisfies~$v \in V_{h'}' = B_G(V_{h'}, r)$, and hence~$H_P$ is a subgraph of~$H'_v$.
    But now~$h \in V(H_P)$ as~$P$ meets the vertex~$w \in V_h$, and~$H_v$ is a subgraph of~$H_P$ as~$v \in V(P)$.
    So since~$H_P$ is connected by~\cref{lem:GraphDecConnSubgraph}, there exists a~$h$--$V(H_v)$~path in~$H_P$ and hence in~$H'_v$, as desired.   
\end{proof}

\subsection{Radial width and radial spread}

While the usual width of a tree-decomposition is measured in terms of the cardinality of its bags, the radial width of a graph-decomposition is measured in terms of the radius of its parts as follows.

\begin{definition}[Radial width] \label{def:radWidth}
	Let $\HdecV$ be a graph-decomposition of a graph~$G$ modelled on a graph~$H$.
	The \emph{(inner-)\radialWidth} of $\HdecV$ is 
	\begin{equation*}
		\rw(H, \cV) \coloneqq \max_{h\in V(H)} \rad(G[V_h]).
	\end{equation*}
    Note that if~$G[V_h]$ is disconnected for some~$h \in H$, then $\rw(H, \cV) = \infty$.
    The \emph{outer-\radialWidth} of $\HdecV$ is $\max_{h\in V(H)} \rad_G(V_h)$.
    We remark that the outer-radial width is at most the (inner-)radial width.
	
	Given a non-empty class $\cH$ of graphs, the \emph{(inner-)\radialHWidth{$\CH$}} of~$G$ is
	\begin{equation*}
		\rHw(G) \coloneqq \min\, \{\radwDec{H, \cV} \mid \HdecV \text{ is a graph-decomposition of }G \text{ with } H \in \CH \}.
	\end{equation*}
    Note that the (inner-)\radialHWidth{$\cH$} will always be at most~$\rad(G)$ for a connected graph~$G$ if $\cH$ contains at least one non-empty graph. The outer-radial $\cH$-width is defined analogously.
\end{definition}

For classes~$\cH$ of graphs that we frequently use in this paper, we name the~\radialHWidth{$\cH$} as in \cref{tab:WidthParameters}.
\begin{table}[ht!]
    \begin{center}
        \caption{Nomenclature for frequently used classes~$\cH$ of graphs. The second column gives the list of graphs which defines the minor-closure of~$\cH$ via forbidden (topological) minors\arXivOrNot{ (cf. \cref{app:CterexQuestionRadWidth})}{}.}
        \label{tab:WidthParameters}
        \begin{tabular}{c | c | c | c}
            \textbf{\shortstack{$\cH$ consists of\\ disjoint unions of}} & \textbf{\shortstack{Forbidden\\ (topological) minors}} & \textbf{$\cH$-decomposition} & \textbf{\radialHWidth{$\cH$}}\\
            \hline
            paths & $K^3$, $K_{1,3}$ & path-decomposition & \radialHWidth{path} \\
            subdivided stars & $K^3$, $W$ & star-decomposition & \radialHWidth{star} \\
            trees & $K^3$ & tree-decomposition & \radialHWidth{tree} \\
            cycles & $K_{1,3}$ & cycle-decomposition & \radialHWidth{cycle}
        \end{tabular}
    \end{center}
\end{table}

Under the name `tree breadth', the concept `outer-radial tree-width' has previously been studied~\cites{dragan2014approximation,leitert2017treebreadth} with applications to tree-spanners and routing problems.

\begin{definition}[Radial spread]
    Let $(H, \cV)$ be a \gd\ of a graph~$G$ modelled on a graph~$H$. 
    The \emph{(inner-)radial spread} of $(H,\cV)$ is 
    \[
    \rs(H, \cV) := \max_{v \in V(G)} \rad(H_v).
    \]
    The \emph{outer-radial spread} of $(H,\cV)$ is 
    $\max_{v \in V(G)} \rad_H(H_v)$.
    Note that the outer-radial spread is at most the (inner-)radial spread. 
\end{definition}

We remark that the `outer' versions of radial width and spread are only used in \cref{lem:TreeBreadthEquiv,lem:GraphDecToQuasiIso,lem:QuasiIsoToGraphDec}. Hence, we will usually omit the `inner' when talking about (inner)-radial width and spread.
\medskip

Berger and Seymour \cite{bergerseymourboundeddiamTD}*{1.5} proved that the (inner-)radial tree-width of a graph is at most two times its outer-radial width, and thus these width measures are qualitatively equivalent. The next lemma generalises this to arbitrary decomposition graphs.

\begin{lemma} \label{lem:TreeBreadthEquiv}
    Let~$(H, \cV)$ be a \gd\ of a graph~$G$ of outer-radial width $k \in \N$ and \mbox{(inner-)} radial spread $r \in \N$.
        Then setting $V_h' := B_G(V_h, k)$ for every node~$h \in H$ yields a \gd\ $(H, \cV')$ of $G$ of (inner-)radial width at most $2k$ and (inner-)radial spread at most $2kr$. 
\end{lemma}
\begin{proof}
    By~\cref{lem:BallsAroundGraphDec},~$(H, \cV')$ is indeed a graph-decomposition of~$H$.
    To show that~$(H, \cV')$ has (inner-)\radialWidth\ at most~$2k$, consider any node~$h \in H$.
    Since $(H,\cV)$ has outer-radial width at most $k$, there is a vertex $z_h$ with $V_h \subseteq B_G(z_h, k)$.
    Then~$z_h \in V_h'$ and~$G[V_h']$ contains a shortest path in~$G$ between every vertex~$x \in V_h$ and~$z_h$.
    Moreover, there exists for any vertex~$x' \in V_h'$ a vertex~$x \in V_h$ with~$\dist_{G[V_h']}(x', x) = \dist_G(x', x) \le k$.
    Thus, we have~$\dist_{G[V_h']}(z_h, x) \le \dist_{G[V_h']}(z_h, x') + \dist_{G[V_h']}(x', x) \le k + k = 2k$.
    This implies that each~$G[V_h']$ has radius at most~$2k$, and hence~$(H, \cV')$ has (inner-)\radialWidth\ at most~$2k$.

    To see that $(H, \cV')$ has (inner-)radial spread at most $2kr$, consider any vertex $v$ of~$G$ and any node~$h$ of~$H$ such that $v \in V'_h$. By the definition of~$V'_h$, there exists a vertex $u \in V_h$ such that $d_G(v,u) \leq k$. Let $P = p_0 \dots p_n$ be a shortest $v$--$u$ path in $G$, in particular, $||P|| = n \leq k$. By~\cref{lem:GraphDecConnSubgraph}, the subgraph $H_P := H[\{h \in H \mid V(P) \cap V_h \neq \emptyset\}]$ of~$H$ is connected. Since $(H, \cV)$ has (inner-)radial spread $r$, the subgraph~$H_P$ has diameter at most $2r \cdot ||P|| \leq 2rk$. As~$P$ starts in~$v$ and ends in $u$, the subgraph~$H_P$ includes~$H_v$ and~$H_u$. Let~$h'$ be the node of $H$ witnessing that~$H_v$ has radius at most~$r$ (in $H_v$); in particular $h' \in H_v \subseteq H_P$. Since also $h \in H_u \subseteq H_P$ and~$H_P$ is connected, there is a path~$Q$ from~$h$ to~$h'$ in~$H_P$ of length at most $2rk$. As $Q \subseteq H_P$ and every bag~$V_g$ for $g \in H_P$ contains a vertex~$w$ of~$P$, which implies $d_G(w, v) \leq ||P|| \leq k$, it follows by the definition of the new bags $V'_h$ that $Q \subseteq H_P \subseteq H'_v := H[\{h \in H \mid v \in V'_h\}]$. Hence, $h'$ witnesses that~$H'_v$ has radius at most~$2rk$, so $(H, \cV')$ has (inner-)radial spread at most $2rk$.
\end{proof}

\subsection{Interplay with quasi-isometries} \label{sec:QuasiIso}

In this section, we describe the interplay between quasi-isometries and honest \gd\ of bounded radial width and spread, i.e.\ we prove \cref{prop:QuasiIsoIntro}. For this, let us recall the definition of quasi-isometries in the case of graphs.

\begin{definition}[Quasi-isometry]
For given integers~$m, M \geq 1$ and~$a, A, r \geq 0$, an \emph{$(m,a,M,A,r)$-quasi-isometry} from a graph~$H$ to a graph~$G$ is a map~$\phi\colon V(H) \to V(G)$ such that

\begin{enumerate}[label=(Q\arabic*)]
    \item \label{quasiisom2.1} $d_H(h, h')\leq m \cdot d_G(\phi(h),\phi(h')) + a$ for every~$h, h' \in V(H)$,
    \item \label{quasiisom2.2} $d_G(\phi(h),\phi(h')) \leq M \cdot d_H(h,h') + A$ for every~$h, h' \in V(H)$, and
    \item \label{quasiisom2.3} for every vertex $v \in G$, there exists a node $h \in H$ with $d_G(v,\phi(h)) \leq r$.
\end{enumerate}
\end{definition}

\noindent Usually~(cf.\,\cite{georgakopoulos2023graph}*{Section~2.1}), quasi-isometries are denoted with only two parameters $L \geq 1$ and $C \geq 0$: an \emph{$(L, C)$-quasi-isometry} is precisely an $(L,LC,L,C,C)$-quasi-isometry.
Here, we use a more detailed set of parameters that allows us to describe in a more precise way how graph-decompositions and quasi-isometries interact.

The following is a well-known fact.

\begin{lemma}\label{lem:InverseQI}
    If a graph $H$ is $(L,C)$-quasi-isometric to a graph $G$, then $G$ is $(L,3LC)$-quasi-isometric to~$H$. \looseness=-1
\end{lemma}

We split the proof of \cref{prop:QuasiIsoIntro} into the following two lemmas.

\begin{lemma} \label{lem:GraphDecToQuasiIso}
    Let $(H,\cV)$ be an honest graph-decomposition of a graph~$G$ of outer-\radialWidth\ $r_0$ and outer-radial spread~$r_1$.
    Then there is a~$(2r_1,2r_1,2r_0,0,r_0)$-quasi-isometry from~$H$ to~$G$.
\end{lemma}

\begin{proof}
    Since~$(H, (V_h)_{h \in H})$ has outer-radial width~$r_0$, there exists for every bag~$V_h$ a vertex~$\phi(h) \in G$ with~$V_h \subseteq B_G(\phi(h), r_0)$.
    We show that the respective map~$\phi: V(H) \to V(G)$ is the desired quasi-isometry from~$H$ to~$G$.
    Note that the definition of~$\phi$ immediately implies~\cref{quasiisom2.3} by~\cref{ax:GraphDec1}.
    
    For~\cref{quasiisom2.2} we have to show that for every~$h, h' \in V(H)$, we have
    \begin{equation*}
        d_G(\phi(h),\phi(h')) \leq 2r_0 \cdot d_H(h,h').
    \end{equation*}
    So let~$h$ and~$h'$ be arbitrary vertices of~$H$, and let~$h_0h_1 \dots h_\ell$ be a shortest~$h$--$h'$ path in~$H$.
    We now aim to build from this path a~$\phi(h)$--$\phi(h')$ walk in~$G$ of length at most~$2 r_0 \ell$.
    For every~$i \in \{1,\dots,\ell\}$, we fix a vertex~$v_i \in V_{h_{i-1}} \cap V_{h_i}$; such~$v_i$ exist because the considered~$H$-decomposition of~$G$ is honest by assumption.
    Furthermore, we set~$v_0 := \phi(h)$ and~$v_{\ell+1} := \phi(h')$.
    Now for every~$i \in \{0, \dots, \ell\}$, let~$P_i$ be a shortest~$v_i$--$v_{i+1}$~path in~$G$.
    Since our~$H$-decomposition of~$G$ has outer-radial width~$r_0$, the paths~$P_0$ and~$P_\ell$ have length at most~$r_0$ and all other~$P_i$ have length at most~$2 r_0$.
    Thus, $v_0P_0v_1 \dots v_\ell P_\ell v_{\ell+1}$ is a~$\phi(h)$--$\phi(h')$~walk in~$G$ of length at most~$r_0 + 2r_0(\ell-1) + r_0$, as desired.

    For~\cref{quasiisom2.1}, we have to check that for every $h, h' \in V(H)$, 
    \begin{equation*}
        d_H(h,h') \leq 2r_1 \cdot d_G(\phi(h),\phi(h')) + 2r_1.  
    \end{equation*}
	So let $h$ and $h'$ be arbitrary vertices of $H$ and let $v_0 v_1 \ldots v_{\ell}$ be a shortest $\varphi(h)$--$\varphi(h')$ path in $G$.
	Similar as before, we build a $h$--$h'$ walk in $H$ of length at most $2 r_1 \ell + 2r_1$.
	For every $i \in \{1, \dots, \ell\}$ we fix a node $h_i\in V(H)$ satisfying $v_{i-1}v_i\in V_{h_i}$;
	such $h_i$ exist by \ref{ax:GraphDec1}.
	Furthermore, we set $h_0\coloneqq h$ and $h_{\ell + 1}\coloneqq h'$.
	For every $i \in \{0, \dots, \ell\}$, let $P_i$ be a shortest $h_{i}$--$h_{i+1}$~path in $H$.
	As $h_{i},h_{i+1}\in V(H_{v_{i}})$ for all $i \in \{0, \dots, \ell\}$ and $(H, \cV)$ has outer-radial spread $r_1$, all $P_i$ have length at most $2 r_1$. 
	Thus, $h_0P_0h_1 \dots h_\ell P_\ell h_{\ell+1}$ is an $h$--$h'$ walk in $H$ of length at most $2r_1(\ell+1)$, as desired.
\end{proof}

\begin{lemma}\label{lem:QuasiIsoToGraphDec}
    Let~$\phi$ be an~$(m,a,M,A,r)$-quasi-isometry from a graph~$H$ to a graph~$G$, and let $r'\geq r$ be an integer.
    Write~$B_h :=  B_G(\phi(h), r')$ and set
    \begin{equation*}
        V_h := \bigcup_{h' \in B_H(h,mr'+ \lceil (m+a)/2 \rceil)} B_{h'}.
    \end{equation*}
    Then~$(H, (V_h)_{h \in H})$ is an honest graph-decomposition of~$G$ of outer-\radialWidth\ at most~$r' + M(mr'+ \lceil (m+a)/2 \rceil) + A$ and (inner-)radial spread at most~$4mr' + m + 2a +1$.
    Moreover, if $r' \geq M+A$, then even the (inner-)\radialWidth\ is at most $4mr' + m + 2a +1$.
\end{lemma}

\begin{proof}
    We first show that the pair~$(H, (V_h)_{h \in H})$ is indeed an~$H$-decomposition of~$G'$.
    It then follows immediately from the definition of the~$V_h$ that $(H, \cV)$ is honest (as $\phi(h) \in V_h$ for all $h \in V(H)$ and $\phi(h), \phi(h') \in V_h \cap V_{h'}$ for all $hh' \in E(H)$ because $m \geq 1$).
    
    \cref{ax:GraphDec1}:
    By~\cref{quasiisom2.3}, each vertex of~$G$ is contained in some~$B_h$.
    Now consider an edge~$e = v_0v_1$ of~$G$.
	There are nodes~$h_0$ and~$h_1$ of~$H$ such that~$v_0 \in B_{h_0}$ and~$v_1 \in B_{h_1}$.
	Hence, $\dist_G(\phi(h_0), \phi(h_1))$ is at most~$2r'+1$.
	By~\cref{quasiisom2.1}, there exists a path~$P$ in~$H$ of length at most~$m(2r'+1)+a$.
	There exists a `middle' vertex~$h$ on this path~$P$ such that both~$h_0$ and~$h_1$ are contained in~$B_H(h, mr'+ \lceil (m+a)/2 \rceil)$, so both~$v_0$ and~$v_1$ are contained in~$B_h \subseteq V_h$.
    In particular, $e = v_0 v_1 \in G[V_h]$, as desired.
	
	\cref{ax:GraphDec2}: 
	Let~$v$ be any vertex of~$G$, and let~$h_0$ and~$h_1$ be nodes of~$H$ such that~$v \in V_{h_0} \cap V_{h_1}$.
	We want to find an~$h_0$--$h_1$ path~$P$ in~$H$ such that~$V_h$ contains~$v$ for every node~$h \in P$.
	By definition of the bags~$V_h$, there are~$h'_0 \in \ball{H}{h_0}{mr'+ \lceil (m+a)/2 \rceil}$ and~$h'_1 \in \ball{H}{h_1}{mr'+ \lceil (m+a)/2 \rceil}$ such that~$v \in B_{h'_0} \cap B_{h'_1}$.
	Hence,~$\dist_G(\phi(h'_0),\phi(h'_1))$ is at most~$2r'$.
	So by~\cref{quasiisom2.1}, there is an~$h'_0$--$h'_1$ path~$P'$ in~$H$ of length at most~$2mr'+a$.
	Let~$W$ be the walk~$P_0 P' P_1$ in~$H$ joining~$h_0$ and~$h_1$ where~$P_0$ is a shortest~$h_0$--$h'_0$ path in~$H$ and~$P_1$ a shortest~$h'_1$--$h_1$ path in~$H$.
	It follows directly from the construction of~$W$ that~$h'_0$ or~$h'_1$ is contained in~$B_H(h, mr'+ \lceil (m+a)/2 \rceil)$ for every node~$h$ visited by~$W$.
    Since~$v \in B_{h'_0} \cap B_{h'_1}$, this implies~$v \in V_h$ for every node~$h$ visited by~$W$ by the definition of~$V_h$.
    In particular, $W$ contains a~$h_0$--$h_1$ path~$P$ which is as desired.
    
    Secondly, let us verify that~$(H, (V_h)_{h \in H})$ has the desired radial spread and outer-radial width.
    For the radial spread, observe that the above constructed walk~$W$ has length at most~$2(mr'+ \lceil (m+a)/2 \rceil) + 2mr'+a \leq 2mr' + m+a + 1 + 2mr' +a = 4mr' + m + 2a +1$; so the radial spread of~$(H,(V_h)_{h \in H})$ is as desired.
    
    For the outer-radial width, consider any node~$h \in H$ and vertex~$v \in V_h$.
    By definition, there is~$h' \in B_H(h,mr'+ \lceil (m+a)/2 \rceil)$ such that~$v \in B_{h'}$.
    By~\cref{quasiisom2.2}, we obtain
    \begin{equation*}
        \dist_G(v,\phi(h)) \leq \dist_G(v,\phi(h')) + \dist_G(\phi(h),\phi(h')) \leq r' + M(mr'+ \lceil (m+a)/2 \rceil) + A.
    \end{equation*}
    Thus, every~$v \in V_h$ has distance at most~$r' + M(mr'+ \lceil (m+a)/2 \rceil) + A$ from~$\phi(h)$.
    This yields the desired outer-radial width.

    To obtain the moreover-part, let us investigate the above equation in more detail.
    Any shortest $v$--$\phi(h')$ in $G$ is contained in $B_h' \subseteq V_h$.
    Fix a shortest $h'$--$h$ path $Q$ and replace each edge $h_0h_1$ of $Q$ by a shortest $\phi(h_0)$--$\phi(h_1)$ path $Q_{h_0h_1}$ in $G$ to obtain a $\phi(h)$--$\phi(h')$ walk $Q'$ in $G$.
    If $r' \geq  M+A$, the path $Q_{h_0h_1}$ is contained in $B'_{h_0} \subseteq V_h$.
    Thus, the (inner-)radial width is already at most $ r' + M(mr'+ \lceil (m+a)/2 \rceil) + A$.
\end{proof}

\begin{proof}[Proof of~\cref{prop:QuasiIsoIntro}]
    Use~\cref{lem:GraphDecToQuasiIso} for \ref{itm:QuasiIsoIntro:GD-Iso} and~\cref{lem:QuasiIsoToGraphDec} for \ref{itm:QuasiIsoIntro:Iso-GD}.
\end{proof}

To conclude this section, let us look at a possible path-way towards omitting the condition on the radial spread.
Berger and Seymour proved that a graph has bounded \radialHWidth{tree} if and only if it is quasi-isometric to a tree~\cite{bergerseymourboundeddiamTD}*{4.1}.
More precisely, they show that if~$G$ has a~$T$-decomposition of low radial-width for some tree~$T$, then $G$ is quasi-isometric to some tree~$T'$ that is obtained from a subtree of~$T$ by contracting and subdividing edges.

We ask whether such an argument transfers to arbitrary decomposition graphs:

\begin{question} \label{qu:QuasiIsoForMinorClosed}
    Given an integer~$r \ge 1$, does there exist an integer~$R$ such that if a graph~$G$ has a decomposition modelled on a graph~$H$ of radial width at most~$r$, then there exists an honest decomposition of~$G$ modelled on a graph $H'$ obtained from a subgraph of $H$ by subdividing and contracting edges such that both its radial width and radial spread are at most~$R$? 
\end{question}

\noindent An affirmative answer to~\cref{qu:QuasiIsoForMinorClosed} would in particular imply the equivalence of small~\radialHWidth{$\cH$} and quasi-isometry to an element of~$\cH$ for graph classes~$\cH$ closed under takings subgraphs, and contracting and subdividing edges.

\subsection{Quasi-geodesic topological minors} \label{subsec:QuasiGeoTopMinors}

In this section, we study quasi-geodesic topological minors as obstructions to small radial width.
We show in~\cref{lem:QuasiGeodTopMinorAreFatMinors} that quasi-geodesic topological minors are also an instance of the more general `fat minors' used in \cref{conj:GeorgakopoulosPapasoglu}.
Fat minors are obstructions to small radial width (\cref{lem:FatMinorObstructRadialWidth}), which then implies that quasi-geodesic topological minors are also such obstructions (\cref{lem:quasigeodesictopologicalminorsareobstructions}).

Let us recall the definition of~$(\geq k)$-subdivisions~$T_k X$ and $c$-quasi-geodesity from the introduction.
\begin{definition}[$T_k X$]
    A \emph{$(\geq k)$-subdivision} of a graph~$X$, which we denote with \emph{$T_k X$}, is a graph which arises from~$X$ by subdividing every edge at least $k$ times, i.e.\ replacing every edge in~$X$ with a new path of length at least $k+1$ such that no new path has an inner vertex in $V(X)$ or on any other new path. 
\end{definition}

\noindent The original vertices of $X$ are the \emph{branch vertices of the $T_k X$}.
Note that the well-known topological minor relation can be phrased in terms of $(\geq 0)$-subdivisions in that a graph $X$ is a \emph{topological minor} of a graph~$G$ if $G$ contains a~$T_0X$ as a subgraph.

\begin{definition}[Quasi-geodesic]
    A subgraph~$X$ of a graph~$G$ is \emph{$c$-quasi-geodesic (in~$G$)} for some~$c \in \N$ if for every two vertices~$u, v \in V(X)$ we have~$\dist_G(u, v) \le c \cdot \dist_X(u, v)$.
    We call~$X$ \emph{geodesic} if it is~$1$-quasi-geodesic. 
\end{definition}

Let us now turn to the question how quasi-geodesic topological minors relate to fat minors, the obstruction investigated by Georgakopoulos and Papasoglu in~\cref{conj:GeorgakopoulosPapasoglu}.
For this, let us first recall the definition of `fat minors'.
Let $G, X$ be graphs.
A \emph{model} $(\cV,\cE)$ of $X$ in $G$ is a collection $\cV$ of disjoint sets $V_x \subseteq V(G)$ for vertices $x$ of $X$ such that each $G[V_x]$ is connected, and a collection $\cE$ of internally disjoint $V_{x_0}$--$V_{x_1}$ paths $E_{e}$ for edges $e=x_0x_1$ of $X$ which are disjoint from every $V_x$ with $x \neq x_0, x_1$.\footnote{Note that we here deviate from the definition of model from \cite{DiestelBook}. However, by enlarging the branch sets $V_x$ along the `adjacent' branch paths $E_{xy}$, we obtain that this notion of model is equivalent to the notion of model from \cite{DiestelBook}.}
The $V_x$ are its \emph{branch sets} and the $E_e$ are its \emph{branch paths}.
Then $X$ is a \emph{minor} of $G$ if $G$ contains a model of $X$.

A model $(\cV, \cE)$ of $X$ in $G$ is \emph{$K$-fat} for $K \in \N$ if $\dist_G(Y,Z) \geq K$ for every two distinct $Y,Z \in \cV \cup \cE$ unless $Y = E_e$ and $Z = V_x$ for some vertex $x \in V(X)$ incident to $e \in E(X)$, or vice versa.
Then $X$ is a \emph{$K$-fat minor} of $G$ if $G$ contains a $K$-fat model of $X$.
We remark that the $0$-fat minors of $G$ are precisely its minors.

\begin{lemma} \label{lem:QuasiGeodTopMinorAreFatMinors}
    If a graph~$G$ contains a~$\subdivk{X}{k}$ as a $c$-quasi-geodesic subgraph for some graph~$X$ and~$c, k \in \N$ with $k \geq 2$, then $X$ is a~$K$-fat minor in~$G$ for~$K = \lfloor \frac{k-2}{2c} \rfloor$.
\end{lemma}
\begin{proof}
    Denote the $c$-quasi-geodesic subgraph of $G$ which is a~$\subdivk{X}{k}$ by $X_G$.
    By definition, any two branch vertices of~$X_G$ have distance at least~$k+1$ in~$X_G$.
    For~$x \in V(X)$, we then choose~$V_x = \ball{X_G}{x}{\lceil(k+1)/4\rceil}$, and for~$e = xx' \in E(X)$, we let~$E_e$ be the (unique)~$V_x$-$V_x'$ path in~$X_G$ not meeting any other branch vertices; in particular, each~$E_e$ has length at least~$(k+1) - 2\lceil(k+1)/4\rceil \ge k/2 - 1$.
    We then let~$\cV$ be the set of all~$V_x$ and~$\cE$ be the set of all~$E_e$.
    Since~$X_G$ is a~$\subdivk{X}{k}$, this construction directly yields~$\dist_{X_G}(Y, Z) \ge k/2-1$ for every two distinct $Y,Z \in \cV \cup \cE$ unless $Y = E_e$ and $Z = V_x$ for some vertex $x \in V(X)$ incident to $e \in E(X)$, or vice versa.
    Since~$X_G$ is a $c$-quasi-geodesic subgraph of~$X$, we then obtain~$\dist_{G}(Y,Z) \ge \dist_{X_G}(Y, Z)/c \ge \frac{k-2}{2c}$; so~$X_G$ is indeed a~$K$-fat model of~$X$ in~$G$.
\end{proof}

\begin{lemma}\label{lem:FatMinorObstructRadialWidth}
    Suppose $X$ is a $K$-fat minor in a graph $G$ for some $K \in \N$. 
    If $(H,\cV)$ is a \gd\ of $G$ of radial width less than $K/2$, then $X$ is a minor of $H$.
\end{lemma}

\begin{proof}
    Fix a $K$-fat model $(\cU,\cE)$ of $X$ in $G$.
    We aim to define a model $(\cU',\cE')$ of $X$ in $H$.
    For every vertex $x \in X$, let $U'_x$ be the set of nodes $h \in H$ whose corresponding bag $V_h$ meets $U_x$.
    Note that the $U'_x$ are pairwise disjoint, as the radial width of $(H,\cV)$ is $< K/2$ and $(\cU, \cE)$ is $K$-fat.
    Since the $G[U_x]$ are connected, the $H[U'_x]$ are connected by \cref{lem:GraphDecConnSubgraph}.
    
    Analogously, for each edge $xy$ of $X$, the subgraph induced on the nodes $h \in H$ whose $V_h$ meets $E_{xy}$ is connected by \cref{lem:GraphDecConnSubgraph}, meets $U'_x$ and $U'_y$, and all these subgraphs are pairwise disjoint.
    Thus, we may pick a $U'_x$--$U'_y$ path $E'_{xy}$ in each of these subgraphs to obtain a model of $X$ in $H$.
\end{proof}

Combining the previous two lemmas, we obtain that quasi-geodesic topological minors are indeed obstructions to small radial width:

\begin{proposition}\label{lem:quasigeodesictopologicalminorsareobstructions}
    If a graph~$G$ contains a~$\subdivk{X}{4ck+2}$ as a $c$-quasi-geodesic subgraph for some graph~$X$ and~$c, k \in \N$ with $k \geq 2$, then~$G$ admits no $H$-decomposition of radial width less than $k$ modelled on a graph~$H$ with no $X$ minor.
    \qed
\end{proposition}

\section{Radial path-width} \label{sec:PathWidth}

In this section we prove \cref{thm:path-width-intro}, which we restate here for convenience. 
\begin{customthm}{\cref*{thm:path-width-intro}} \label{thm:pathWidth}
	Let~$k \in \N$.
    If a connected graph~$G$ contains no~$\subdivk{K^3}{4k+1}$ as a geodesic subgraph and no~$\subdivk{K_{1,3}}{3k}$ as a~$3$-quasi-geodesic subgraph, then~$G$ admits an honest decomposition modelled on a path $P$ of radial width at most~$18k + 2$ and radial spread at most~$18k+1$.

    Moreover, $P$ is $(1, 18k+2)$-quasi-isometric to~$G$.
\end{customthm}

\noindent We remark that we did not optimise the bounds on the radial width and radial spread.
\medskip

In fact, we show the following stronger statement, which immediately implies~\cref{thm:pathWidth}.

\begin{theorem} \label{thm:RadialPathWith:Technical}
    Let~$k \in \N$, and let $P$ be a longest geodesic path in a connected graph~$G$. If~$G$ contains no~$\subdivk{K^3}{4k+1}$ as a geodesic subgraph and no $\subdivk{K_{1,3}}{3k}$ as a~$3$-quasi-geodesic subgraph, then either $P$ has length at most $18k+2$ or every vertex of $G$ has distance at most $9k$ from $P$.
\end{theorem}

Let us first show that \cref{thm:RadialPathWith:Technical} indeed implies \cref{thm:pathWidth}. For this, we need the following auxiliary lemma, which asserts that taking balls of radius $r \in \N$ around a quasi-geodesic subgraph $H$ of a graph $G$ yields a `partial decomposition' of $G$ modelled on $H$ (see \cref{fig:ball-decomp}).

\begin{lemma} \label{lem:UnionOfBallIsDecomp}
    Let $G$ be a graph, and let $H$ be a $c$-quasi-geodesic subgraph of $G$ for some~$c \in \N$.
    Given~$r \in \N$ we write~$B_h := B_G(h, r)$ for~$h \in H$ and set
    \begin{equation*}
        V_h := \bigcup_{h' \in B_H(h, cr + \lceil c/2 \rceil)} B_{h'}.
    \end{equation*}
    Then $(H, (V_h)_{h \in H})$ is an honest~$H$-decomposition of $G' := G[\bigcup_{h \in H} B_h]$ of {\radialWidth} at most $(c+1) r + \lceil c/2 \rceil$ and radial spread at most $4cr+c+1$.

    Moreover, if $H$ is a path, then $(H, (V_h)_{h \in H})$ has radial spread at most $(c+1)r+\lceil c/2 \rceil$.
\end{lemma}
\begin{proof}
    The pair~$(H, (V_h)_{h \in H})$ is an honest graph-decomposition of~$G$ by~\cref{lem:QuasiIsoToGraphDec}, since the embedding~$\phi$ of a~$c$-quasi-geodesic subgraph~$H$ of~$G$ into~$G[\bigcup_{h \in H} B_G(h, r)]$ is a~$(c, 0, 1, 0, r)$-quasi-isometry.
    By~\cref{lem:QuasiIsoToGraphDec} and its moreover-part, this \gd\ has (inner-)radial width~$r + (c r+ \lceil c/2 \rceil) = (c+1) r + \lceil c/2 \rceil$ and radial spread $4rc+c+1$.
        
    For the `moreover'-part, assume that $H = h_0 \dots h_n$ is a path, and let $v$ be any vertex of $G$. Let $i, j \in [n]$ be the smallest and largest integer such that $v \in B_{h_k} = B_G(h_k, r)$ for $k = i,j$. Then $d_G(h_i, h_j) \leq 2r$, and hence $j-i \leq 2cr$ since $H$ is $c$-quasi-geodesic in $G$. By the definition of the $V_h$, it follows that only~$V_{h_k}$ with $i-cr+\lceil c/2\rceil \leq k \leq j+cr+ \lceil c/2\rceil$ contain $v$, and hence $H_v$ has radius at most $(c+1)r + \lceil c/2\rceil$ as it is contained in $B_H(h_k, (c+1)r + \lceil c/2\rceil)$ for $k := \lfloor (j-i)/2\rfloor$.
\end{proof}

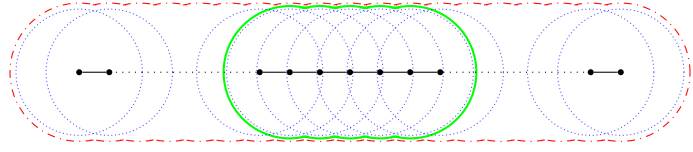
\begin{figure}[h]
    \centering
    \begin{tikzpicture}[scale=0.4,smallVertex/.style={inner sep=1}]
        \node[smallVertex] (a9) at (-9,0) {};
        \node[smallVertex] (a8) at (-8,0) {};
        \node[smallVertex] (a7) at (-7,0) {};
        \node[smallVertex] (a6) at (-6,0) {};
        \node[smallVertex] (a5) at (-5,0) {};
        \node[smallVertex] (a4) at (-4,0) {};
        \node[smallVertex] (a3) at (-3,0) {};
        \node[smallVertex] (a2) at (-2,0) {};
        \node[smallVertex] (a1) at (-1,0) {};
        \node[smallVertex] (a0) at (0,0) {};
        \node[smallVertex] (b9) at (9,0) {};
        \node[smallVertex] (b8) at (8,0) {};
        \node[smallVertex] (b7) at (7,0) {};
        \node[smallVertex] (b6) at (6,0) {};
        \node[smallVertex] (b5) at (5,0) {};
        \node[smallVertex] (b4) at (4,0) {};
        \node[smallVertex] (b3) at (3,0) {};
        \node[smallVertex] (b2) at (2,0) {};
        \node[smallVertex] (b1) at (1,0) {};

                \foreach \x in {a0,a1,a2,a3,a4,a5,a6,a7,a8,a9,b1,b2,b3,b4,b5,b6,b7,b8,b9}{
            \draw[red, dashdotted] (\x) circle (2.3);
        }
        \foreach \x in {a0,a1,a2,a3,a4,a5,a6,a7,a8,a9,b1,b2,b3,b4,b5,b6,b7,b8,b9}{
            \fill[white] (\x) circle (2.25);
        }

                \foreach \x in {a0,a1,a2,b1,b2}{
            \draw[thick, green] (\x) circle (2.2);
        }
        \foreach \x in {a0,a1,a2,b1,b2}{
            \fill[white] (\x) circle (2.15);
        }
                \foreach \x in {a0,a1,a2,a3,a8,a9,b1,b2,b3,b8,b9}{
            \draw[blue, opacity=0.7, densely dotted] (\x) circle (2.1);
        }
        
                \foreach \x in {a0,a1,a2,a3,a8,a9,b1,b2,b3,b8,b9}{
            \fill[black] (\x) circle (0.1);
        }

                \draw[black] (a0) -- (a1) -- (a2) -- (a3)  (a8) -- (a9);
        \draw[black] (a0) -- (b1) -- (b2) -- (b3)  (b8) -- (b9);
        \draw[black,dotted] (a3) -- (a8) (b3) -- (b8);
    \end{tikzpicture}
    
    \caption{A decomposition modelled on the black path using the blue balls as given by \cref{lem:UnionOfBallIsDecomp}
    for $r=2$, $c=1$. The bag corresponding to the centre vertex is green.}
    \label{fig:ball-decomp}
\end{figure}

\begin{proof}[Proof of \cref{thm:pathWidth} given \cref{thm:RadialPathWith:Technical}]
    Let $G$ be a connected graph that contains neither $\subdivk{K_{1,3}}{3k}$ as a $3$-geodesic subgraph nor $\subdivk{K^3}{4k+1}$ as a geodesic subgraph. Let~$P'$ be a longest geodesic path in~$G$. By \cref{thm:RadialPathWith:Technical}, either~$P'$ has length at most $18k+2$ or every vertex of~$G$ has distance at most $9k$ from~$P'$. 
    In the former case, it follows that~$G$ has radius at most $18k+2$. Let~$P$ be the trivial path on a single vertex~$p$. Then~$P$ is $(1, 18k+2)$-quasi-isometric to~$G$, and $(P, (V_p)_{p \in P})$ with $V_p := V(G)$ is an honest decomposition of~$G$ of radial width at most $18k+2$ and radial spread~$0$.

    So we may assume the latter case. Set $P := P'$, and apply \cref{lem:UnionOfBallIsDecomp} to $H := P$ and $r := 9k$. This yields an honest $P$-decomposition of $G' := G[B_G(P, 9k)]$ of radial width at most $18k + 1$ and radial spread at most $18k+1$. Since every vertex of $G$ has distance at most $9k$ from $P$, we have $G' = G$, and hence this is the desired decomposition of $G$.
\end{proof}

The remainder of this section is devoted to the proof of \cref{thm:RadialPathWith:Technical}. 
Let us first give a brief sketch of the proof. For this, let~$G$ be a connected graph and, let~$P$ be a longest geodesic path in $G$. For some suitably chosen~$r = r(k) < 9k$, we let $G_P := G[B_G(P, r)]$. We then analyse how the components of~$G - G_P$ attach to $G_P$, and show that either all vertices in a component have distance at most $9k$ from $P$ or we can use the component to find a~$\subdivk{K_{1,3}}{3k}$ as a~$3$-quasi-geodesic subgraph of~$G$ or a~$\subdivk{K^3}{3k}$ as a geodesic subgraph of~$G$.

The analysis of the components will be done in three lemmas, \cref{lem:CombiningGeodesicGraphs,lem:FindingALongGeodesicCycle,lem:ComponentsAttachingToTheEndOfP} below. They are stated in a slightly more general form than needed for the proof of~\cref{thm:RadialPathWith:Technical}, which enables us to use them later also in the proofs of~\cref{thm:cycle-width-intro,thm:star-width-intro}.

The first lemma, \cref{lem:CombiningGeodesicGraphs} shows that enlarging a $c$-quasi-geodesic subgraph with a shortest path to it yields a subgraph which is~$(2c+1)$-quasi-geodesic. This lets us find a $3$-quasi-geodesic $\subdivk{K_{1,3}}{3k}$ in~$G$ if some component of $G-G_P$ attaches `to the middle' of~$G_P$, that is, to some ball $B_G(p, r)$ where $p$ lies `in the middle' of $P$.
The second lemma, \cref{lem:FindingALongGeodesicCycle}, demonstrates that we can find a long geodesic cycle in~$G$ if some component of $G-G_P$ attaches to $G'$ close to the start and the end of $P$, but nowhere in between (see \cref{fig:banana-case}).
The third lemma, \cref{lem:ComponentsAttachingToTheEndOfP}, shows that if a component of $G-G_P$ attaches to $G_P$ only towards one end of $P$, then its vertices all have distance at most $9k$ from $P$.

\begin{lemma} \label{lem:CombiningGeodesicGraphs}
    Let $G$ be a graph, and let $X$ be a $c$-quasi-geodesic subgraph of $G$ for some~$c \in \N$.
    If~$P$ is a shortest~$v$--$X$ path in $G$ for some vertex~$v \in G$, then~$X \cup P$ is~$(2c+1)$-quasi-geodesic in~$G$.
\end{lemma}

\begin{proof}
	We have to show that, for every two vertices $u$ and~$w$ of $X \cup P$, the distance of $u$ and $w$ in $X \cup P$ is at most $(2c+1)$ times their distance in $G$.
	Since~$X$ is~$c$-quasi-geodesic in~$G$ and~$P$ is geodesic in~$G$, it is (by symmetry) enough to consider the case where $u$ is a vertex of $P$ and $w$ is a vertex of $X$.
    Let $x$ be the endvertex of~$P$ in~$X$.
	Since $P$ is a shortest $v$--$X$ path in $G$, $uPx$ is a shortest~$u$--$X$ path in~$G$ and hence $\dist_P(u, x) = \dist_G(u,x) \leq \dist_G(u,w)$ as~$w \in X$.
    We then have~$\dist_X(x, w) \le c \cdot \dist_G(x, w) \le c \cdot (\dist_G(x, u) + \dist_G(u, w))$, where the first inequality follows from~$X$ being $c$-quasi-geodesic in~$G$ while the second one applies the triangle inequality.
	Again using the triangle inequality, we can then combine these inequalities to
    \begin{align*}
         \dist_{X \cup P}(u, w) &\le \dist_{X \cup P}(u, x) + \dist_{X \cup P}(x, w) = \dist_P(u, x) + \dist_X(x, w) \\
         &\le \dist_G(u, x) + c \cdot (\dist_G(x, u) + \dist_G(u, w)) \le (2c + 1) \cdot \dist_G(u, w),
    \end{align*}
	which shows the claim.
\end{proof}

\begin{lemma} \label{lem:FindingALongGeodesicCycle}
	Let~$r, n, m_0, m_1 \in \N$ such that $m := n - m_0 - m_1 - 2r > 0$.
	Let~$G$ be a graph containing a geodesic path~$P = p_0 \dots p_n$ of length~$n$, and write~$B_i := B_G(p_i, r)$ for every~$p_i \in P$.
	Suppose that a component~$C$ of~$G - \bigcup_{i = 0}^n B_i$ has at least one neighbour in both~$\bigcup_{i = 0}^{m_0} B_i$ and~$\bigcup_{i = n - m_1}^n B_i$, but no neighbours in~$\bigcup_{i = m_0+1}^{n-m_1-1} B_i$.
	Then $G$ contains a geodesic cycle of length at least $2m$.
\end{lemma}

\begin{figure}[ht]
    \centering
    \begin{tikzpicture}[scale=0.4,smallVertex/.style={inner sep=1}]
        \node[smallVertex] (a9) at (-9,0) {};
        \node[smallVertex] (a8) at (-8,0) {};
        \node[smallVertex] (a7) at (-7,0) {};
        \node[smallVertex] (a6) at (-6,0) {};
        \node[smallVertex] (a5) at (-5,0) {};
        \node[smallVertex] (a4) at (-4,0) {};
        \node[smallVertex] (a3) at (-3,0) {};
        \node[smallVertex] (a2) at (-2,0) {};
        \node[smallVertex] (a1) at (-1,0) {};
        \node[smallVertex] (a0) at (0,0) {};
        \node[smallVertex] (b9) at (9,0) {};
        \node[smallVertex] (b8) at (8,0) {};
        \node[smallVertex] (b7) at (7,0) {};
        \node[smallVertex] (b6) at (6,0) {};
        \node[smallVertex] (b5) at (5,0) {};
        \node[smallVertex] (b4) at (4,0) {};
        \node[smallVertex] (b3) at (3,0) {};
        \node[smallVertex] (b2) at (2,0) {};
        \node[smallVertex] (b1) at (1,0) {};

                \foreach \x in {a0,a1,a2,a3,a4,a5,a6,a7,a8,a9,b1,b2,b3,b4,b5,b6,b7,b8,b9}{
            \draw[red, dashdotted] (\x) circle (2.3);
        }
        \foreach \x in {a0,a1,a2,a3,a4,a5,a6,a7,a8,a9,b1,b2,b3,b4,b5,b6,b7,b8,b9}{
            \fill[white] (\x) circle (2.25);
        }

                \foreach \x in {a6,a7,a8,a9,b6,b7,b8,b9}{
            \draw[blue, opacity=0.7, densely dotted] (\x) circle (2.1);
        }
        
                \foreach \x in {a6,a7,a8,a9,b6,b7,b8,b9}{
            \fill[black] (\x) circle (0.1);
        }

                \draw[black] (a9) -- (a8) -- (a7) -- (a6);
        \draw[black] (b9) -- (b8) -- (b7) -- (b6);
        \draw[black,dotted] (a6) -- (b6);

                \draw[densely dashed] (-5.5,3.5)
            to[out=45, in=135] (5.5,3.5)
            to[out=315, in=225] (8,3)
            to[out=45, in=315] (8,5)
            to[out=135, in=45] (-8,5)
            to[out=225, in=135] (-8,3)
            to[out=315, in=225] (-5.5,3.5);

                \node[smallVertex] (A) at (-7.5,3) {};
        \node[smallVertex] (A2) at (-7,3) {};
        \node[smallVertex] (B) at (-7.5,1.9) {};
        \node[smallVertex] (B2) at (-7,2) {};
        \node[smallVertex] (B3) at (-6.5,1.9) {};
        \node[smallVertex] (B4) at (-8,2) {};
        \node[smallVertex] (C) at (7.5,3) {};
        \node[smallVertex] (C2) at (7,3) {};
        \node[smallVertex] (C3) at (6.5,3) {};
        \node[smallVertex] (D) at (7.5,1.9) {};
        \node[smallVertex] (D2) at (7,2) {};
        \node[smallVertex] (D3) at (6.5,1.9) {};
        \node[smallVertex] (D4) at (8,2) {};
        \foreach \x in {A,A2,B,B2,B3,B4,C,C2,C3,D,D2,D3,D4}{
            \fill[black] (\x) circle (0.1);
        }
        \draw (B3) -- (A) -- (B) (A) -- (B2) -- (A2) (B4) -- (A2) -- (B) (D4) -- (C) -- (D) -- (C3) (C) -- (D2) -- (C3) -- (D3) (D4) -- (C2) -- (D2);
    \end{tikzpicture}
    \caption{The setting of \cref{lem:FindingALongGeodesicCycle}. 
    }
    \label{fig:banana-case}
\end{figure}

\noindent See \cref{fig:banana-case} for a sketch of the situation in \cref{lem:FindingALongGeodesicCycle}.
\medskip

The proof of~\cref{lem:FindingALongGeodesicCycle} builds on the study of how cycles interact with a given separation of the graph.
More formally, we have the following preparatory lemma.

\begin{figure}
	\begin{tikzpicture}
		\draw (0,0) -- (6,0) -- (6,4) -- (5,4) -- (5,0) -- (11,0) -- (11,4) -- (0,4) -- (0,0) (5,2) -- (6,2);
		\path[pattern=north east lines, pattern color=orange] (5,0) -- (6,0) -- (6,2) -- (5,2) -- (5,0);
		\path[pattern=north west lines, pattern color=green] (5,4) -- (6,4) -- (6,2) -- (5,2) -- (5,4);
		\node at (5.5,4.5) {$S$};
		\node at (8.5,4.5) {$C$};
		\node at (5.5,1) {$M_0$};
		\node at (5.5,3) {$M_1$};
		\draw[densely dotted, black] (4,1.5) -- (7,2.5) node[pos=.25,below,black] {$Q$};
		\draw[dashed, purple] (4,1.5) to[out=-130, in=-90] (3,2) to[out=90, in=180] (5.1,3.1) to[out=0,in=0] (5.1,3.4) to[out=180, in=-90] (4.8,3.55) to[out=90, in=180] (5.1,3.7) to (7,3.7) to[out=0,in=0] (7,3.9) to[out=180,in=90] (2.5,2) to[out=-90,in=180] (5.1,1.25) to[out=0,in=-50] (7,2.5);
		\draw[dashdotted, teal] (7,2.5) to[out=130, in=130] (7.5,2.5) to[out=-50, in=0] (5.5,0.8) to[out=180, in=-90] (2,1.8) to[out=90, in=90] (1.5,1.8) to[out=-90, in=180] (5.5,0.4) to[out=0, in=-90] (8.5,2) to[out=90,in=20] (6,3) to[out=200, in=0] (5.5,2.5) to[out=180,in=50] (4,1.5);
            \node at (8.5,3) {$O_1$};
	\end{tikzpicture}
	\caption{The setting of \cref{lem:DiscreteWindingNumber}. Here, $O_2$ is the cycle containing $Q$ (black, dotted line) and the purple, dashed part of $O_1$, and $O_3$ is the cycle containing $Q$ and the teal, dash dotted part of $O_1$.}
	\label{fig:paritycycle}
\end{figure}
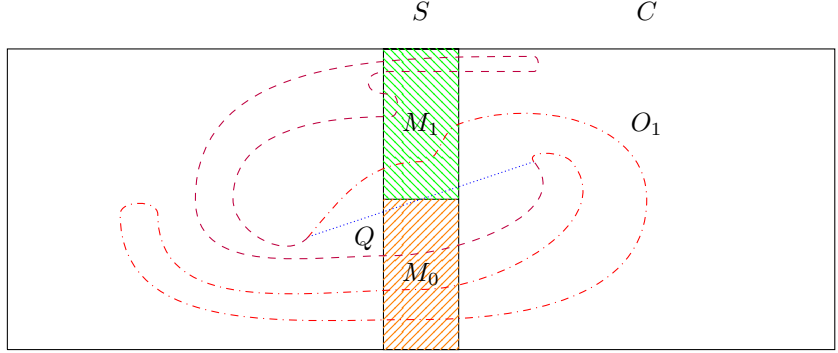

\begin{lemma} \label{lem:DiscreteWindingNumber}
    Let $G$ be a graph, $S$ a set of vertices and $C$ a component of $G - S$, and $\{M_0,M_1\}$ be a bipartition of $S$.
    Given a cycle $O$ in $G$, let $W(O)$ be the number of ~$M_0$--$M_1$ paths in~$O$ meeting~$C$.
    
	Suppose that~$Q$ is an~$O_1$-path of a cycle~$O_1$ in~$G$, and let~$O_2$ and~$O_3$ denote the two cycles in~$O_1 \cup Q$ containing~$Q$.
	Then~$W(O_1) + W(O_2) + W(O_3)$ is even.
\end{lemma}
\begin{proof}
    The reader may look at \cref{fig:paritycycle} to follow the proof more easily.
	Let~$u$ and~$v$ be the endvertices of~$Q$ on~$O_1$.
	Write~$P_1 := Q$, and let~$P_2$ and~$P_3$ be the two~$u$--$v$ paths in~$O_1$.
	Then the cycle~$O_1$ consists of~$P_2$ and~$P_3$, and without loss of generality, we may assume that~$O_2$ consists of $P_1$ and $P_2$ while $O_3$ consists of~$P_1$ and $P_3$.
	
	Let~$\cP$ be the set of $M_0$--$M_1$ paths in $G$ meeting~$C$ that are paths in at least one $O_i$.
	We remark that every path in~$\cP$ has all its inner vertices in~$C$ since~$S = M_0 \cup M_1$ separates~$C$ from the rest of the graph.
	Moreover, let us emphasise that, in the definition of~$W(O)$, we do not consider the direction in which~$O$ traverses the~$M_0$--$M_1$ paths it contains.
	
	If $T \in \cP$ is a subpath of some $P_i$, then $T$ is also a subpath of exactly two of the three cycles $O_1$, $O_2$ and $O_3$ and thus contributes exactly $2$ to the sum $W(O_1) + W(O_2) + W(O_3)$.
	So when we check that this sum is even, all the elements of $\cP$ that are a subpath of some $P_i$ contribute an even amount to the sum.
	In particular, if all elements of $\cP$ are paths in some $P_i$, then the claim holds.
	
	So let $\cP' \subseteq \cP$ consist of precisely those elements of $\cP$ that are not a path in any $P_i$, and suppose that~$\cP'$ is non-empty.
    This implies that at least one $P_i$ meets $M_0 \cup M_1$.
    Every path in $\cP$ is a subpath of one of the cycles $O_i$, which in turn is disjoint from the interior of some $P_i$.
    Hence every path in $\cP'$ meets the interior of at most, and thus precisely, $2$ of the~$P_i$.
	We remark that a path in an $O_i$ that is also a path in some $O_j$ with $i \neq j$ is already a path in some $P_i$.
	Thus, every element of $\cP'$ contributes exactly $1$ to the sum $W(O_1) + W(O_2) + W(O_3)$.
    In order to prove that this sum is even, it thus suffices to show that the number of paths in~$\cP'$ is even.
    We will check this by a case distinction.
	Note that a path in some $O_j$ is also a path in some $P_i$ if and only if it contains neither of $u$ and $v$ as an inner vertex.
    Hence, all paths in~$\cP'$ contain at least one of~$u$ and~$v$ as an inner vertex.
    Because $\cP'$ is non-empty and all inner vertices of its elements are contained in~$C$, at least one of $u$ and $v$ is contained in $C$.

	First, consider the case that all three paths $P_1$, $P_2$ and $P_3$ meet $M_0 \cup M_1$. In particular, this implies that no path in $\cP'$ contains both $u$ and $v$ as an inner vertex.
    If all paths $P_1$, $P_2$ and $P_3$ have their first vertex in $M_0 \cup M_1$ contained in the same element of $\{M_0, M_1\}$, then no element of~$\cP'$ contains~$u$ as an inner vertex. 
    Otherwise, precisely two paths $P_i$ have their first vertex in $M_0 \cup M_1$ contained in the same element of $\{M_0, M_1\}$. 
    In this case, there are exactly two elements of $\cP'$ that contain $u$ as an inner vertex.
    By symmetry, there are also exactly $2$ or $0$ elements of $\cP'$ that contain $v$ as an inner vertex. 
    Hence, $\cP'$ contains precisely $0,2$ or $4$ paths, and thus in this case $W(O_1) + W(O_2) + W(O_3)$ is even.
    
	Secondly, consider the case that precisely one $P_i$, say $P_1$, meets $M_0\cup M_1$.
    Since $P_2$ does not meet $M_0 \cup M_1$ and at least one of $u$ and $v$ is contained in $C$, this implies that both $u$ and $v$ are contained in $C$.
    Because $\cP'$ is non-empty, the first and last vertex of~$P_1$ in $M_0 \cup M_1$ are contained in distinct elements of $\{M_0, M_1\}$.
    Thus, $\cP'$ contains precisely two paths, one containing $P_2$ and one containing $P_3$.
	Hence, $W(O_1) + W(O_2) + W(O_3)$ is even.
	
	Lastly, consider the case that precisely two $P_i$ meet $M_0 \cup M_1$. 
	Say $P_3$ does not meet $M_0 \cup M_1$.
	Thus, again both $u$ and $v$ are contained in $C$.
	Here, we need to distinguish some more cases.
	
	The first case which we consider is that both $P_1$ and $P_2$ have their first and last vertex in $M_0 \cup M_1$ contained in distinct elements of $\{M_0,M_1\}$ and that the first vertex of $P_1$ in $M_0 \cup M_1$ and the first vertex of $P_2$ in $M_0 \cup M_1$ are contained in distinct elements of $\{M_0,M_1\}$.
	Then $\cP'$ contains precisely four elements: two containing $P_3$ and two not containing any edges of $P_3$.
	So again $W(O_1) + W(O_2) + W(O_3)$ is even.
	
	Next, we assume that both $P_1$ and $P_2$ have their first and last vertex in $M_0 \cup M_1$ contained in distinct elements of $\{M_0,M_1\}$ but that the first vertex of $P_1$ in $M_0 \cup M_1$ and the first vertex of $P_2$ in $M_0 \cup M_1$ are contained in the same element of $\{M_0,M_1\}$.
	Then $\cP'$ contains precisely two elements: both contain $P_3$.
	So again $W(O_1) + W(O_2) + W(O_3)$ is even.
	
	Now we assume that at least one of $P_1$ and $P_2$, say $P_1$, has their first and last vertex in $M_0 \cup M_1$ contained in the same element of $\{M_0,M_1\}$, say in $M_0$.
	Since $\cP'$ is non-empty and every path in $\cP'$ contains~$u$ or $v$ as inner vertex, this implies that either the first or the last vertex of~$P_2$ in $M_0 \cup M_1$ is contained in~$M_1$.
	If only the first or only the last vertex of $P_2$ in $M_0 \cup M_1$ is contained in $M_1$, then $\cP'$ contains precisely two elements: one of these contains~$P_3$ and the other does not contain edges of $P_3$.
	If both the first and the last vertex of $P_2$ in $M_0 \cup M_1$ are contained in $M_1$, then $\cP'$ contains also precisely two elements: neither of them contains edges of $P_3$.
	Again in both cases, $W(O_1) + W(O_2) + W(O_3)$ is even.
\end{proof}

\begin{proof}[Proof of~\cref{lem:FindingALongGeodesicCycle}]
	Note that, as $n > m_0 + m_1$, we in particular have $n - m_1 > m_0$.
	We set 
    \begin{equation*}
        M_0:= N_G(C) \cap \bigcup_{i = 0}^{m_0} B_i \text{ and } M_1 := N_G(C) \cap \bigcup_{i = n - m_1}^n B_i.
    \end{equation*}
	As $C$ has no neighbours in $\bigcup_{i = m_0+1}^{n-m_1-1} B_i$, the neighbourhood of~$C$ is equal to $M_0 \cup M_1$ and $C$ is a component of $G - M_0 - M_1$.
	
	First, we note that $M_0$ and $M_1$ are disjoint.
	Indeed, if $u\in B_{j_0}$ for some $j_0 \leq m_0$ and $v\in B_{j_1}$ for some~$j_1 \geq n-m_1$, then using the triangle inequality we obtain
	\begin{equation}
		\label{eq:distanceM0M1}
		n - m_0 - m_1 \leq \dist_P(p_{j_0},p_{j_1}) = \dist_G(p_{j_0},p_{j_1}) \leq \dist_G(p_{j_0},u) + \dist_G(u,v) + \dist_G(v,p_{j_1}) = d_G(u,v) + 2r.
	\end{equation}
	So $\dist_G(u,v) \geq n - m_0 - m_1 - 2r = m > 0$ by assumption.
	In particular, $u$ and $v$ are distinct vertices, and
	thus~$M_0$ and~$M_1$ are disjoint.
	
	We now define, for a cycle $O$ in $G$, the number $W(O)$ as the number of $M_0$--$M_1$ paths in $O$ which have an inner vertex in $C$.
	
	The remainder of the proof now consists of three steps.
	First, we show that a cycle~$O$ with~$W(O) \neq 0$ has length at least~$2m$.
    Second, we find a cycle~$O$ with odd~$W(O)$, and lastly we show that a shortest such cycle is geodesic in~$G$.
	
	\begin{claim} \label{cl:GeoCycLength}
		Every cycle $O$ in~$G$ with $W(O) \neq 0$ has length at least $2m$.
	\end{claim}
	\begin{claimproof}
		Since $W(O) \neq 0$, the cycle~$O$ contains at least one $M_0$--$M_1$ path.
		As the number of $M_0$--$M_1$ paths is even for every cycle in~$G$, the cycle~$O$ contains at least two $M_0$--$M_1$ paths, which then have to be internally disjoint.
        By \cref{eq:distanceM0M1}, every such path has length at least $m$, and thus $O$ has length at least $2m$.
	\end{claimproof}
	
	\begin{claim} \label{cl:GeoCycOddExists}
		There exists a cycle~$O$ in~$G$ with odd~$W(O)$.
	\end{claim}
	\begin{claimproof}
		Among all $M_0$--$M_1$ paths that contain an inner vertex in $C$, let $Q$ be a shortest such path, and denote by $v_0$ and $v_1$ its end vertices in $M_0$ and $M_1$, respectively.
		Let $R_0$ be a shortest $v_0$--$P$ path in $G$ and~$R_1$ a shortest $v_1$--$P$ path.
		Since $v_i \in M_0 \cup M_1 = N_G(C)$, both $R_i$ have length exactly $r$.
		If $R_0$ and~$R_1$ share a vertex~$x$, say such that $R_0x$ is at least as long as $R_1x$, then $R_1xR_0$ contains a path of length at most~$r$ from $v_1$ to some $p_s$ with $s\leq m_0$, contradicting $v_1 \in M_1$ and $M_0 \cap M_1 = \emptyset$.
		Hence, $R_0$ and~$R_1$ are disjoint.
		Now the concatenation $R_0 Q R_1$ is a path that starts and ends in distinct vertices of $P$ and is internally disjoint from $P$. 
		Hence, it can be closed by a (unique) subpath of~$P$ to a cycle $O$.
		By construction of $O$, there is precisely one $M_0$--$M_1$ path in $O$ that has an inner vertex in~$C$, and that is~$Q$.
		In particular, $W(O)$ is odd.
	\end{claimproof}
	
	\begin{claim} \label{cl:GeoCycShortenOdd}
		Let $O$ be a cycle in~$G$ with odd $W(O)$.
		If $O$ is not geodesic, then there is a cycle $O'$ that is shorter than $O$ such that $W(O')$ is odd.
	\end{claim}
	\begin{claimproof}
		Since $O_1 := O$ is not geodesic, there is an~$O_1$-path $Q$ in $G$ with endvertices $u$ and $v$ in $O$ which is shorter than the distance of $u$ and $v$ in $O_1$.
		Let~$O_2$ and~$O_3$ be the two cycles in~$O_1 \cup Q$ which contain~$Q$.
		Since both $O_2$ and $O_3$ are shorter than $O_1$, it suffices to show that at least one of $W(O_2)$ and $W(O_3)$ is odd.
		Now~$W(O_1)$ is odd by assumption, so it is enough to show that the sum~$W(O_1) + W(O_2) + W(O_3)$ is even.
        This however follows from~\cref{lem:DiscreteWindingNumber} applied to the component $C$ together with the bipartition~$\{M_0, M_1\}$ of~$S=N_G(C)$.
	\end{claimproof}
	
	To formally complete the proof, pick a cycle $O$ in~$G$ with $W(O)$ odd such that $O$ is as short as possible among these.
	Such~$O$ exists by~\cref{cl:GeoCycOddExists} and is geodesic in~$G$ by~\cref{cl:GeoCycShortenOdd}.
	Moreover, $O$ has length at least~$2m$ by~\cref{cl:GeoCycLength}, as desired.
\end{proof}

\begin{lemma} \label{lem:ComponentsAttachingToTheEndOfP}
    Let~$U$ be a set of vertices of a connected graph~$G$, and let~$p_0 \in V(G) \setminus U$ have maximal distance from~$U$ in~$G$.
    Let~$P = p_0 \dots p_n$ be a shortest~$p_0$--$U$ path in~$G$.
    Given~$r \in \N$, we write \mbox{$B_i := B_G(p_i, r)$} for every $p_i \in P$. 
    Suppose that some component~$C$ of~$G - \bigcup_{i = 0}^n B_i$ is disjoint from~$U$ and satisfies \mbox{$N_G(C) \subseteq \bigcup_{i = 0}^\ell B_i$} for some~$\ell \le n$.
    Then~$\dist_G(v, P) \leq 2r + \ell$ for every~$v \in C$.
\end{lemma}

\begin{proof}
	Let $v$ be an arbitrary vertex of $C$, and let $Q$ be a shortest $v$--$U$ path in $G$ with endvertex $u \in U$.
	As $V(C) \cap U = \emptyset$, the path $Q$ intersects $\bigcup_{i = 0}^n B_i$.
	Let $q$ be the first vertex on $Q$ in $\bigcup_{i = 0}^n B_i$, and let $j$ be the smallest index of a ball~$B_j$ with $q \in B_j$.
	Note that $j \leq \ell$ since~$N_G(C) \subseteq \bigcup_{i = 0}^\ell B_i$.
	
	Since $P$ is a longest geodesic path in $G$ which ends in $U$, we have that
	$n = ||P|| \ge ||Q|| = \dist_G(v, u) = \dist_G(v, q) + \dist_G(q, u)$, where the last equality follows from~$Q$ being geodesic in $G$.
    Additionally, since~$p_n \in U$, we have $\dist_G(p_j, p_n) \leq \dist_G(p_j, u) \leq \dist_G(p_j, q) + \dist_G(q, u) = r + \dist_G(q, u)$.
    These two inequalities combine to
    \begin{equation*}
        \dist_G(v, q) \leq n - \dist_G(q, u) \leq n - (\dist_G(p_j, p_n) - r).
    \end{equation*}
	All in all, we get
	\begin{equation*}
	    \dist_G(v, P) \leq \dist_G(v, p_j) \leq \dist_G(v, q) + r \leq n - (\dist_G(p_j, p_n) - r) + r = n - (n - j) + 2r \leq 2r + \ell,
	\end{equation*}
    where the last inequality follows from~$j \leq \ell$.
    This concludes the proof of the claim.
\end{proof}

With all the previous lemmas at hand, we are now ready to prove~\cref{thm:RadialPathWith:Technical}. 
\begin{proof}[Proof of \cref{thm:RadialPathWith:Technical}]
	Let $G$ be a connected graph that contains neither $\subdivk{K_{1,3}}{3k}$ as a~$3$-quasi-geodesic subgraph nor~$\subdivk{K^3}{4k+1}$ as a geodesic subgraph. 	
	Let $P = p_0 \dots p_n$ be a longest geodesic path in~$G$. We may assume $n \geq 18k + 3$; otherwise we are done.
    	For every $0 \leq i  \leq n$, define $B_i := B_G(p_i, 3k)$ as the ball in~$G$ of radius $3k$ around $p_i$. 		
	If~$G = G_P$, then we are done.
    Otherwise, we consider the components of $G - G_P$. Since $G$ is connected by assumption, each component~$C$ of~$G - G_P$ has a neighbour in at least one $B_i$.
    The following two claims show that in fact $N_G(C) \subseteq \bigcup_{i=0}^{3k} B_i$ or $N_G(C) \subseteq \bigcup_{i=n-3k}^n B_i$.

            	\begin{claim} \label{path:AttachmentsInTheMiddle}
            No component $C$ of $G - G_P$ has a neighbour in $B_i$ with $3k < i < n - 3k$.
	\end{claim}

	\begin{claimproof}
		Suppose for a contradiction that there is such a component $C$.
        We show that $G$ then contains a~$\subdivk{K_{1,3}}{3k}$ as $3$-quasi-geodesic subgraph, which contradicts our assumption on $G$.
        
		Let $v \in C$ be a vertex with a neighbour in $B_i$ with~$3k < i < n-3k$, and let $Q$ be a shortest $v$--$p_i$ path in~$G$.
		The choice of~$v$ guarantees that the path $Q$ has length exactly $3k+1$, and thus $Q$ is a shortest $v$--$P$~path in $G$ as every vertex in $G - G_P$ and hence in $C$ has distance at least $3k+1$ to $P$.
		By \cref{lem:CombiningGeodesicGraphs}, $P \cup Q$ is a $3$-quasi-geodesic subgraph of $G$.
		It follows from $3k < i < n - 3k$ that $P \cup Q$ is a $\subdivk{K_{1,3}}{3k}$, as desired.
	\end{claimproof}
	
	\begin{claim}\label{path:AttachmentsOnBothSides}
		No component $C$ of $G - G_P$ has at least one neighbour in $\bigcup_{i=0}^{3k} B_i$, at least one neighbour in $\bigcup_{i=n-3k}^n B_i$ and no neighbours in any $B_i$ with $3k < i < n-3k$.
	\end{claim}

	\begin{claimproof}
        Suppose for a contradiction that there is such a component $C$.
        Applying~\cref{lem:FindingALongGeodesicCycle} to~$C$ with~$m_0 = m_1 = r = 3k$, we find that $G$ contains a geodesic cycle~$O$ of length at least~$2 (n - 12k)$.
        But~$n \geq 18 k + 3$ as we assumed the graph to have a large radius, so~$O$ has length at least~$2 (18k + 3 - 12k) = 12k + 6$.
        Thus, $O$ is a geodesic cycle in~$G$ of length at least $12k + 6$, a contradiction to our assumptions on~$G$.
	\end{claimproof}

	Thus, each component of $G - G_P$ attaches either only to balls~$B_i$ with~$i \le 3k$ or to balls~$B_i$ with~$i \ge n - 3k$.
	But the vertices in such components have bounded distance from $P$ in $G$, as the following claim shows.

	\begin{claim}\label{path:AttachmentsOnOneSide}
		Let $C$ be a component of $G - G_P$ such that $N_G(C) \subseteq \bigcup_{i=0}^{3k} B_i$ or $N_G(C) \subseteq \bigcup_{i=n-3k}^n B_i$.
		Then every vertex $v \in C$ has distance at most $9 k$ from $P$ in $G$.
	\end{claim}

	\begin{claimproof}
        If~$N_G(C) \subseteq \bigcup_{i=0}^{3k} B_i$, then the claim follows by applying \cref{lem:ComponentsAttachingToTheEndOfP} with $U = \{p_n\}$.
        Otherwise, $N_G(C) \subseteq \bigcup_{i=n-3k}^n B_i$, and the claim follows by applying \cref{lem:ComponentsAttachingToTheEndOfP} with $U = \{p_0\}$.
	\end{claimproof}

	By \cref{path:AttachmentsInTheMiddle,path:AttachmentsOnOneSide,path:AttachmentsOnBothSides}, every vertex of $G$ has distance at most $9k$ from $P$, as desired.
    \end{proof}

\section{Radial cycle-width} \label{sec:CycleWidth}

In this section we build on our result on~\radialHWidth{path}, \cref{thm:pathWidth}, and use a small adaptation of its proof  to address~\radialHWidth{cycle} by proving~\cref{thm:cycle-width-intro}.
More precisely, \cref{thm:cycle-width-intro} follows from the lemmas that we have already shown in~\cref{sec:PathWidth}.
Let us restate~\cref{thm:cycle-width-intro} here for convenience. 
\begin{customthm}{\cref*{thm:cycle-width-intro}} \label{thm:cycleWidth}
    Let~$k \in \N$.
    If a connected graph~$G$ contains no~$\subdivk{K_{1,3}}{3k}$ as a~$3$-quasi-geodesic subgraph, then~$G$ admits an honest decomposition modelled on a path or cycle $C$ of radial width at most~$18k+2$ and radial spread at most~$36k+2$.
    
    Moreover, $C$ is $(1, 18k+2)$-quasi-isometric to $G$.
\end{customthm}

\noindent We remark that we did not optimise the bound on the radial width and radial spread.

In fact, we show the following stronger statement, which immediately implies \cref{thm:cycleWidth}.

\begin{theorem} \label{thm:RadialCycleWidth:Technical}
    Let $k \in \N$, and let $G$ be a connected graph. If $G$ contains no~$\subdivk{K_{1,3}}{3k}$ as a~$3$-quasi-geodesic subgraph, then there exists a geodesic cycle or path~$C$ in $G$ such that $C$ is either a path of length at most $18k+2$ or every vertex of $G$ has distance at most $9k$ from $C$. 
\end{theorem}

Let us first show that \cref{thm:RadialCycleWidth:Technical} implies \cref{thm:cycleWidth}.

\begin{proof}[Proof of \cref{thm:cycleWidth} given \cref{thm:RadialCycleWidth:Technical}]
    Let $G$ be a connected graph that contains no~$\subdivk{K_{1,3}}{3k}$ as a~$3$-quasi-geodesic subgraph, and let~$C'$ be given by applying \cref{thm:RadialCycleWidth:Technical} to~$G$. 
    If $C'$ is a path of length at most $18k+2$, then $G$ has radius at most $18k+2$. Let $C$ be the trivial path on a single vertex $c$. Then $C$ is $(1, 18k+2)$-quasi-isometric to $G$, and $(C, (V_c))$ with $V_c := V(G)$ is an honest decomposition of $G$ of radial width at most $18k+2$ and radial spread $0$.

    So we may assume that every vertex of $G$ has distance at most $9k$ from $C$. Set $C := C'$, and apply \cref{lem:UnionOfBallIsDecomp} to $H := C$ and $r := 9k$. This yields an honest $C$-decomposition of $G' := G[B_G(C, 9k)]$ of radial width at most $18k + 1$ and radial spread at most $36k + 2$. Since every vertex of $G$ has distance at most $9k$ from~$C$, we have $G' = G$, and hence this is the desired decomposition of~$G$.
\end{proof}

Let us now prove \cref{thm:RadialCycleWidth:Technical}.

\begin{proof}[Proof of \cref{thm:RadialCycleWidth:Technical}]
	Let~$G$ be a connected graph that contains no $\subdivk{K_{1,3}}{3k}$ as a $3$-quasi-geodesic subgraph. 	Applying \cref{thm:RadialPathWith:Technical} to $G$ yields that $G$ either contains a geodesic path $P$ such that $C := P$ is as desired, or $G$ contains a $\subdivk{K^3}{4k+1}$ as a geodesic subgraph.
	    Since we are done in the first case, we may thus assume that~$G$ contains a $\subdivk{K^3}{4k+1}$ as a geodesic subgraph, which means that there exists a geodesic cycle~$C$ in~$G$ of length at least $3\cdot(4k + 2) = 12k + 6$.
    Set $G_C := G[B_G(C, 3k)]$. If $G = G_C$, then we are done.

		    Otherwise, there exists a neighbour~$v$ of~$V(G_C)$ in~$G$ since $G$ is connected.
	Let $P$ be a shortest $v$--$C$ path in $G$, and let $u$ be its endvertex in $C$.
	Since $v \in N_G(G_C)$, it has distance $3k + 1$ from $C$, and thus $P$ has length $3k+1$.
	Let $Q := B_C(u, 3k + 1)$ be the subpath of~$C$ of length $6k+2$ which contains~$u$ as its `middle' vertex.
	Note that~$Q$ is actually a path as~$C$ has length at least $12k + 6 > 6k + 2$.
	Since $P$ is a path from $v$ to $C$ which ends in~$u$, the paths~$P$ and $Q$ only meet in $u$.
	Hence, $P \cup Q$ is a $\subdivk{K_{1,3}}{3k}$ in $G$.
	
	To conclude the proof and obtain the desired contradiction to our assumptions on $G$, it remains to show that $P \cup Q$ is $3$-quasi-geodesic in $G$.
	Since $Q$ is a subpath of $C$ with $||Q|| = 6k + 2 \leq \frac{1}{2}(12k + 6) \leq \frac{1}{2}||C||$, we have $\dist_Q(u,v) = \dist_C(u,v)$ for every two vertices $u,v\in Q$.
    Thus, $Q$ is geodesic in~$G$, since $C$ is geodesic in $G$.
	By \cref{lem:CombiningGeodesicGraphs}, $P \cup Q$ then is a $3$-quasi-geodesic subgraph of $G$, as desired.
\end{proof}

\section{Radial star-width} \label{sec:StarWidth}

In this section we prove \cref{thm:star-width-intro}, which we restate here for convenience. 
\begin{customthm}{\cref*{thm:star-width-intro}} \label{thm:starWidth}
	Let~$k \in \N$.
    If a connected graph~$G$ contains no~$\subdivk{K^3}{k}$ as a geodesic subgraph and no~$\subdivk{W}{3k}$ as a~$3$-quasi-geodesic subgraph, then $G$ admits an honest decomposition modelled on a subdivided star of radial width at most~$58k + 9$ and radial spread at most $30k+7$. 
    Moreover, there exists some $C_k \in \N$ such that some subdivided star is $(1, C_k)$-quasi-isometric to $G$.
\end{customthm}

\noindent Recall that $W$ is the {\em wrench\/} graph depicted in \cref{fig:wrench}. As we already mentioned in the introduction, one can check by carefully reading the proof that the subdivided star which we construct for the first part of the statement already satisfies the `moreover'-part, and that we may choose $C_k = 60k+14$ 
(see the paragraph after the proof of \cref{thm:starWidth} for details).
\medskip

Before we start with the proof of \cref{thm:starWidth}, we first show an auxiliary lemma. Recall that in the proof of \cref{thm:pathWidth} we used \cref{lem:FindingALongGeodesicCycle} as a tool to show that a graph contains a geodesic $\subdivk{K^3}{k}$.
Similarly, we will use the following lemma in the proof of \cref{thm:starWidth} to show that a graph contains a $3$-quasi-geodesic~$\subdivk{W}{3k}$.

\begin{lemma} \label{lem:FindingA3geodesicW}
	Let~$P = p_0 \dots p_n$ be a geodesic path in a graph~$G$.
	Given two vertices~$u, v \in G$, let~$Q$ be a shortest~$u$--$P$ path in~$G$ with endvertex~$q \in P$, and let~$R$ be a shortest~$v$--$(P \cup Q)$ path in~$G$ with endvertex~$r \in P \cup Q$.
	Suppose that at least one of the following two conditions hold:
	\begin{enumerate}[label=(\roman*)]
		\item \label{CaseP} $r \in P$ and~$\dist_G(q, r) \geq 4\max\{\dist_G(u, P), \dist_G(v, P)\}$;
		\item \label{CaseQ} $r \in Q$ and~$\dist_G(q, r) \geq 4\max\{\dist_G(v, Q), \dist_G(p_0, q), \dist_G(p_n, q)\}$.
	\end{enumerate}
	Then $P \cup Q \cup R$ is a $3$-quasi-geodesic subgraph of $G$.
\end{lemma}

\begin{proof}
	We first consider the case that $r \in P$.
	By \cref{lem:CombiningGeodesicGraphs}, we have that $P \cup Q$  and $P \cup R$ are each $3$-quasi-geodesic subgraphs of $G$.
	Hence, in order to show that $X := P \cup Q \cup R$ is $3$-quasi-geodesic, it suffices to consider vertices $x, y \in X$ with $x \in Q$ and $y \in R$.
	
	Since $P$ is geodesic, we have that $\dist_X(q, r) = \dist_G(q, r) \leq \dist_G(q, x) + \dist_G(x, y) + \dist_G(y, r)$ and thus 
	\begin{align*}
		\dist_G(x, y) &\geq \dist_G(q, r) - \dist_G(x, q) - \dist_G(y, r) \geq 4\max\{\dist_G(u, P), \dist_G(v, P)\} - \dist_G(u, P) - \dist_G(v, P)\\
		&\geq 2\max\{\dist_G(u, P), \dist_G(v, P)\} \geq 2\max\{\dist_G(x, q), \dist_G(y, r)\}.
	\end{align*}
	Therefore, making again use of $\dist_G(q, r) \leq \dist_G(q, x) + \dist_G(x, y) + \dist_G(y, r)$, we obtain
	\begin{align*}
		\dist_X(x,y) &= \dist_X(x, q) + \dist_X(q,r) + \dist_X(r, y) = \dist_G(x, q) + \dist_G(q, r) + \dist_G(r, y)\\
		&\leq 2\dist_G(x, q) + 2\dist_G(y, r) + \dist_G(x, y) \le 3\dist_G(x, y).
	\end{align*}
	
	For the second case, assume that $r \in Q$.
    Similar as before, we find that $P \cup Q$ and $Q \cup R$ are each $3$-quasi-geodesic subgraphs of $G$.
	Hence, we only need to consider vertices $x, y \in X$ with $x \in P$ and $y \in R$.
    As before, we find that 
	\[
	\dist_G(x, y) \geq \dist_G(q, r) - \dist_G(x, q) - \dist_G(y, r) \geq 2\max\{\dist_G(v, Q), \dist_G(p_0, q), \dist_G(p_n, q)\},
	\]
	and thus 
	\[
	\dist_X(x, y) = \dist_G(x, q) + \dist_G(q, r) + \dist_G(r, y) \leq 2\dist_G(x, q) + 2\dist_G(y, r) + \dist_G(x, y) \leq 3\dist_G(x, y).\qedhere
	\]
\end{proof}

Now we prove \cref{thm:starWidth}.
We remark that we did not optimise the bounds on the radial width and radial spread.

\begin{proof}[Proof of \cref{thm:starWidth}]
	Let~$G$ be a graph that contains neither $\subdivk{K^3}{k}$ as a geodesic subgraph nor $\subdivk{W}{3k}$ as a $3$-quasi-geodesic subgraph.
	We will construct a subdivided star $S$ and an $S$-decomposition of~$G$ of radial width at most $58k + 9$ and radial spread at most $58k+9$.
    The `moreover' part of the statement then follows: 
    By \cref{lem:GraphDecToQuasiIso,lem:InverseQI}, there exists an $(L, C)$-quasi-isometry from $G$ to $S$ such that $L$ and $C$ depend on~$k$ only. Applying \cite{nguyenAsymptoticStructureII2025}*{1.3} (to $L, C, 2$)\footnote{Note that we use here that subdivided stars admit path-decompositions into bags of size~$3$ and thus have path-width $2$.} then yields a constant $C'_k$ such that $G$ is $(1, C'_k)$-quasi-isometric to some subdivided star $S'$. By \cref{lem:InverseQI}, it then follows that $S'$ is $(1, C_k)$-quasi-isometric to~$G$ for $C_k := 3C'_k$.
    	    \medskip
	
		Let $P = p_0 \ldots p_n$ be a longest geodesic path in $G$.
	Observe that we may assume $n \geq 58k + 10$; otherwise,~$G$ has a trivial $\cH$-decomposition into a single ball of radius at most $58k + 9$.
	For every $0 \leq i \leq n$ define $B_i$ as the ball in~$G$ of radius $3k$ around~$p_i$, and let $G_P:= G[\bigcup_{i = 0}^n B_i]$.
	
	In the following claims we will analyse how the components of $G - G_P$ attach to $G_P$.
    Note that, as we assumed $G$ to be connected, every component of $G - G_P$ has some neighbour in some $B_i$.
	First we show that we may assume that some component of $G-G_P$ attaches to $G_P$ somewhere in the middle of $P$.
	
	\begin{claim} \label{starwidthclaim1}
		Either $G$ has an honest $P$-decomposition of radial width at most $58k + 9$ and radial spread at most $58k + 9$ or some vertex of $G-G_P$ has a neighbour in some $B_s$ with $23k + 5 \leq s \leq n - 23k - 5$.
	\end{claim}
	
	\begin{claimproof}
		We assume for a contradiction that whenever a component of $G - G_P$ has a neighbour in some~$B_s$ then $s\leq 23k + 4$ or $n - 23k - 4 \leq s$.
		
		First we consider the case that some component~$C$ of $G - G_P$ has at least one neighbour in both $\bigcup_{i = 0}^{23k + 4} B_i$ and $\bigcup_{i = n - 23k - 4}^n B_i$.
        Let $m_0 = m_1 = 23k + 4$ and $r = 3k$.
        As
        \[
        n - m_0 - m_1 - 2r = n - (23k + 4) - (23k + 4) - 6k = n - 52k - 8 \geq 2k + 2 \geq 2,
        \]
        we can apply \cref{lem:FindingALongGeodesicCycle} to obtain a geodesic cycle in~$G$ of length at least $2\cdot(2k + 2) > 3k + 3$ which contradicts our assumption that $G$ does not contain a geodesic $\subdivk{K^3}{k}$.
		
		Thus, for every component $C$ of $G - G_P$, either $N_G(C) \subseteq \bigcup_{i = 0}^{23k + 4} B_i$ or $N_G(C) \subseteq \bigcup_{i = n - 23k - 4}^n B_i$ (this includes in particular the case that $G = G_P$).
		Then by \cref{lem:ComponentsAttachingToTheEndOfP} applied to $U = \{p_0\}$ or $U = \{p_n\}$, every vertex in $G - G_P$ has distance at most $2 \cdot 3k + 23k + 4 = 29k + 4$ to~$P$.
		Hence, by defining $B_i'=B_G(p_i,29k + 4)$, every vertex of $G$ is contained in some $B_i'$.
		So by \cref{lem:UnionOfBallIsDecomp} applied to the geodesic path $P$ and to $r= 29k + 4$ we obtain an honest $P$-decomposition of $G$ of radial width at most $2\cdot(29k + 4) + 1 = 58k + 9$ and radial spread at most $2\cdot(29k + 4) + 1 = 58k + 9$.
	\end{claimproof}
	
	By \cref{starwidthclaim1} (and since $P$ is a path and hence a subdivided star) we may assume that there is some $s$ with $23k + 5 \leq s \leq n - 23k - 5$ such that some vertex of $G - G_P$ has a neighbour in $B_s$. 
    We fix $s$ for the rest of the proof.
	The following two claims now show that every component of $G - G_P$ can only attach to~$G_P$ either close to the start of~$P$ or close to the end of~$P$ or close to~$p_s$.
	
	\begin{claim} \label{starwidthclaim2}
		If a vertex of $G - G_P$ has a neighbour in a ball $B_t$, then $t\leq 3k$ or $s - 12k - 3\leq t \leq s + 12k + 3$ or $n - 3k \leq t$.
	\end{claim}
	
	\begin{claimproof}
		We assume for a contradiction that there is some vertex $v$ of $G - G_P$ that has a neighbour in some~$B_t$ with $3k < t < s - 12k - 3$, the other case is symmetric. Let $R = r_0 \ldots r_{3k + 1} $ be a shortest path from $v=r_0$ to $p_t = r_{3k + 1}$ in $G$.
		Similarly, let $u$ be a vertex of $G - G_P$ that has a neighbour in $B_s$ and let $Q = q_0 q_1 \ldots q_{3k + 1}$ be a shortest path from $u=q_0$ to $p_s = q_{3k + 1}$ in $G$.
		Then $p_tPp_s$ has length at least~$12k + 4 = 4\cdot(3k+1)$, so by applying \cref{lem:FindingA3geodesicW} to $p_{t - 3k - 1}Pp_{s + 3k + 1}$, $Q$ and $R$ we obtain a $3$-quasi-geodesic~$\subdivk{W}{3k}$ in~$G$, a contradiction to our assumption on~$G$.
	\end{claimproof}
	
	\begin{claim}\label{starwidthclaim3}
		For every component $C$ of $G - G_P$, either $N_G(C) \subseteq \bigcup_{i = 0}^{3k} B_i$ or $N_G(C) \subseteq \bigcup_{i = s - 12k - 3}^{s + 12k + 3} B_i =: D_s$ or $N_G(C) \subseteq \bigcup_{i = n - 3k}^n B_i$.
	\end{claim}
	
	\begin{claimproof}
		First, we assume for a contradiction that $C$ has neighbours in both $\bigcup_{i = 0}^{3k} B_i$ and $D_s$.
        Let $r= 3k$, $m_0 = 3k$ and $m_1 = n - (s - 12k - 3)$.
        Then, since $s \geq 23k+5$,
        \[
        n - m_0 - m_1 - 2r = n - 3k - (n - (s - 12k - 3)) - 6k = s - 21k - 3 \geq 2k + 2 \geq 2,
        \]
		so we can apply \cref{lem:FindingALongGeodesicCycle} to obtain a geodesic cycle in~$G$ of length at least $2\cdot(2k + 2) > 3k + 3$.
		This contradicts our assumption that $G$ does not contain a geodesic $\subdivk{K^3}{k}$.
		
		Hence, if $C$ has neighbours in $D_s$, then it does not have neighbours in $\bigcup_{i = 0}^{3k} B_i$, and by symmetry it neither has neighbours in $\bigcup_{i = n - 3k}^n B_i$.
		Furthermore, if $C$ has neighbours in both $\bigcup_{i = 0}^{3k} B_i$ and $\bigcup_{i = n - 3k}^n B_i$, then it follows again from \cref{lem:FindingALongGeodesicCycle} that $G$ contains a geodesic $\subdivk{K^3}{k}$, which yields the same contradiction.
	\end{claimproof}

    Let us first consider all the components of $G - G_P$ that attach to $G_P$ either close to the start or close to the end of $P$. By \cref{lem:ComponentsAttachingToTheEndOfP} these components can only contain vertices which are close to $P$. 
	This fact allows us to find a path-decomposition of $G' := G[V']$ of low radial width and spread where $V'$ is the union of the $B_i$ and the vertices of components $C$ of $G - G_P$ with $N_G(C) \subseteq \bigcup_{i = 0}^{3k} B_i$ or $N_G(C) \subseteq \bigcup_{i = n - 3k}^n B_i$, as follows.
	
	\begin{claim}\label{starwidthclaim4}
    All vertices of $G'$ have distance at most $9k$ from $P$.
    Also, there is an honest decomposition $(P, \cV')$ 
    of~$G'$ of radial width at most $24k + 5$ and radial spread at most $24k+5$ such that
    \begin{itemize}
        \item all vertices in the bag $V_s'$ of the node $p_s$ of $P$ have distance at most $24k + 5$ from $p_s$ in $G[V'_s]$, and
                \item for every $v \in N_G(G')$ the set $N_G(v) \cap V'$ is contained in the bag $V_s'$.
    \end{itemize}
	\end{claim}
	
	\begin{claimproof}
		First, we show that vertices of $G' - G_P$ are close to $p_0$ or $p_n$.
		Let $C$ be a component of $G' - G_P$ with~$N_G(C) \subseteq \bigcup_{i = 0}^{3k} B_i$.
		By \cref{lem:ComponentsAttachingToTheEndOfP} applied to $U = \{p_n\}$ and $\ell = 3k$, every vertex $v \in C$ has distance at most $2\cdot (3k) + 3k = 9k$ to $P$, and thus has distance at most $12k$ to $p_0$.
		By symmetry, every component~$C$ of~$G' - G_P$ with $N_G(C) \subseteq \bigcup_{i = n - 3k}^n B_i$ only contains vertices whose distance from~$p_n$ is at most $12k$.
		
		Now we construct the $P$-decomposition.
		For $i\in [n]$ define
		\[
        V'_i := \bigcup_{p_j \in B_P(p_i, 12k + 3)} B_{G'}(p_j,12k + 2).
        \]
		By \cref{lem:UnionOfBallIsDecomp}, the $V'_i$ are the bags of a $P$-decomposition of radial spread at most $2(12k+2) + 1 = 24k + 5$, and every element of $V'_i$ has distance at most~$24k + 5$ from $p_i$ in $G[V'_i]$ by construction.
		Furthermore, every component of $G' - G_P$ is contained in $B_{G'}(p_0,12k + 2) \subseteq V'_0$ or in $B_{G'}(p_n,12k + 2) \subseteq V'_n$.
  
		Also, for every $B_i$ that contains a neighbour of $G - G' \subseteq G - G_P$ we have $s - 12k - 3 \leq i \leq s + 12k + 3$ by \cref{starwidthclaim2} and thus $B_i \subseteq B_{G'}(p_i,12k + 2) \subseteq V'_s$.
		As $V'_s$ is the bag corresponding to $p_s$, this completes the proof.
	\end{claimproof}
    
	We now construct path-decompositions of low radial width and spread of the remaining components, that is of the components of $G - V'$.
    We then combine these path-decompositions with the path-decomposition $(P, \cV')$ of~$G'$ to a star-decomposition of~$G$, whose central node will be~$p_s$. 
    For this, we need to enlarge the bag $V'_s$ assigned to~$p_s$ a little.
    Indeed, at the moment, the components of $G - G'$ need not have low radial path-width; in fact, they can still be star-like.
	But if we delete larger balls around the nodes in $P$ that are close to $p_s$, we indeed end up with components that are path-like.
	For this, we define $V''$ to be the set of vertices of $G$ that have distance at most $18k + 4$ from $P$, and let $G'':=G[V'']$ (so $G' \subseteq G''$). In particular, by \cref{starwidthclaim2}, every component of $G-G''$ `attaches in the middle of $P$', i.e.\ every neighbour of $G''$ in a component of $G-G''$ has distance $18k+5$ to a $p_i$ with $s-12k-3 \leq i \leq s + 12k+3$. Let further $\cC$ be the set of components that contain a vertex whose distance from $P$ is at least $24k+5$, and let $\cC'$ be the set of all other components of $G-V''$. We now add all vertices in $V''\setminus V'$ and all components from $\cC'$ to the bag $V'_s$ that is indexed by $p_s$. More precisely, we define
    \[
	V_s := V_s' \cup (V'' \setminus V') \cup \{V(C) \colon C \in \cC'\}.
	\]
    Further, we set $V_i := V'_i$ for all $i \neq s \in [n]$. Let also $V''' := V''\cup V_s$ and $G''' := G[V''']$. 

    \begin{claim} \label{claim:Starwidth:New}
        $(P, \cV)$ is a decomposition of $G'''$ of radial width at most $36k+7$ and radial spread at most $24k+5$ and every vertex in $V_s$ has distance at most $36k+7$ from $p_s$ in $G[V_s]$.
    \end{claim}
    
    \begin{claimproof}
        Since $N_G(v) \cap V'$ is contained in $V'_s \subseteq V_s$ for every vertex $v \in N_G(G')$ by \cref{starwidthclaim4}, and because $(P, \cV')$ is a decomposition of $G'$, it follows that $(P, \cV)$ is a decomposition of $G'''$. Its radial spread is at most $24k+5$ by \cref{starwidthclaim4} and because each vertex in $G'''- V'$ is only contained in $V_s$. Moreover, every part $G[V_i]$ with $i \neq s$ has radius at most $24k+5$ by \cref{starwidthclaim4}, so it remains to consider $G[V_s]$. 

        Let $v \in V_s$ be given. If $v \in V'_s$, then $d_{G[V_s]}(v, p_s) \leq 24k+5$ by \cref{starwidthclaim4}. Otherwise, $v$ has distance at most $24k+4$ from some vertex $p_i$ of $P$, where $s-12k-3 \leq i \leq s+12k+3$ by \cref{starwidthclaim2}, and hence
        \[
        d_{G[V_s]}(v, p_s) \leq d_{G[V_s]}(v, p_i) + d_{G[V_s]}(p_i, p_s) \leq (24k+4) + (12k+3) = 36k + 7
        \]
        where we used that $p_iPp_s \subseteq G[V'_s] \subseteq G[V_s]$ by the definition of $V'_s$.
    \end{claimproof}

    We now show that all components of $G-V'''$, i.e.\ those in $\cC$, are path-like.	
	\begin{claim}\label{starwidthclaim5}
        Let $C$ be a component in $\cC$.
		Then there is an honest decomposition of $G[V(C) \cup V_s]$ modelled on a path $Q$ of radial width at most $54k+8$ and radial spread at most $12k+2$ such that the bag corresponding to the last node of $Q$ contains $V_s$, and $V_s$ only meets bags assigned to nodes of distance at most $6k$ to the last node of $Q$.
            
												            			\end{claim}

	\begin{figure}
		\centering
		\begin{tikzpicture}[scale=0.4,smallVertex/.style={inner sep=1}]
		\node[smallVertex] (a9) at (-9,0) {};
		\node[smallVertex] (a8) at (-8,0) {};
		\node[smallVertex] (a7) at (-7,0) {};
		\node[smallVertex] (a6) at (-6,0) {};
		\node[smallVertex] (a5) at (-5,0) {};
		\node[smallVertex] (a4) at (-4,0) {};
		\node[smallVertex] (a3) at (-3,0) {};
		\node[smallVertex] (a2) at (-2,0) {};
		\node[smallVertex] (a1) at (-1,0) {};
		\node[smallVertex] (a0) at (0,0) {};
		\node[smallVertex] (b9) at (9,0) {};
		\node[smallVertex] (b8) at (8,0) {};
		\node[smallVertex] (b7) at (7,0) {};
		\node[smallVertex] (b6) at (6,0) {};
		\node[smallVertex] (b5) at (5,0) {};
		\node[smallVertex] (b4) at (4,0) {};
		\node[smallVertex] (b3) at (3,0) {};
		\node[smallVertex] (b2) at (2,0) {};
		\node[smallVertex] (b1) at (1,0) {};
		\node[smallVertex] (q9) at (0,1) {};
		\node[smallVertex] (q8) at (0,2) {};
		\node[smallVertex] (q7) at (0,3) {};
		\node[smallVertex] (q6) at (0,4) {};
		\node[smallVertex] (q5) at (0,5) {};
		\node[smallVertex] (q4) at (0,6) {};
		\node[smallVertex] (q3) at (0,7) {};
		\node[smallVertex] (q2) at (0,8) {};
		\node[smallVertex] (q1) at (0,9) {};
		\node[smallVertex] (q0) at (0,10) {};
		\node[smallVertex] (q01) at (0,11) {};
		\node[smallVertex] (q02) at (0,12) {};
		\node[smallVertex] (q03) at (0,13) {};
		\node[smallVertex] (q04) at (0,14) {};
		\node[smallVertex] (q05) at (0,15) {};
		\node[smallVertex] (q06) at (0,16) {};
		
				\foreach \x/\start/\ende in {a9/-14/-8.5,a8/-8.5/-7.5,a7/-7.5/-6.5,a6/-6.5/-5.5,a5/-5.5/-4.5,a4/-4.5/-3.5,a3/-3.5/-2.5,a2/-2.5/-1.5,a1/-1.5/-0.5,a0/-0.5/0.5,b1/0.5/1.5,b2/1.5/2.5,b3/2.5/3.5,b4/3.5/4.5,b5/4.5/5.5,b6/5.5/6.5,b7/6.5/7.5,b8/7.5/8.5,b9/8.5/14}{
			\begin{scope}
				\clip (\start,5) rectangle (\ende,-5);
				\draw[violet, thick, densely dotted] (\x) circle (4);
			\end{scope}
		}

				\draw[thick, purple, dotted] 
		\foreach \x in {a0,a1,a2,a3,a4,a5,a6,b1,b2,b3,b4,b5,b6}{
			(\x) circle (1.8)
		}
		\foreach \x in {a7,a8,a9,b7,b8,b9}{
			(\x) circle (2.8)
		};
		\foreach \x in {a0,a1,a2,a3,a4,a5,a6,b1,b2,b3,b4,b5,b6}{
			\fill[white] (\x) circle (1.75);
		}
		\foreach \x in {a7,a8,a9,b7,b8,b9}{
			\fill[white] (\x) circle (2.75);
		}

				\foreach \x/\start/\ende in {a9/-11/-8.5,a8/-8.5/-7.5,a7/-7.5/-6.5,a6/-6.5/-5.5,a5/-5.5/-4.5,a4/-4.5/-3.5,a3/-3.5/-2.5,a2/-2.5/-1.5,a1/-1.5/-0.5,a0/-0.5/0.5,b1/0.5/1.5,b2/1.5/2.5,b3/2.5/3.5,b4/3.5/4.5,b5/4.5/5.5,b6/5.5/6.5,b7/6.5/7.5,b8/7.5/8.5,b9/8.5/11}{
			\begin{scope}
				\clip (\start,2) rectangle (\ende,-2);
				\draw[thick, red, loosely dotted] (\x) circle (1.75);
			\end{scope}
		}
	
				\foreach \x/\start/\ende in {a0/-2/0.5,q9/0.5/1.5,q8/1.5/2.5,q7/2.5/3.5,q6/3.5/4.5,q5/4.5/5.5,q4/5.5/6.5,q3/6.5/7.5,q2/7.5/8.5,q1/8.5/9.5,q0/9.5/10.5,q01/10.5/11.5,q02/11.5/12.5,q03/12.5/13.5,q04/13.5/14.5,q05/14.5/15.5,q06/15.5/16.3}{
			\begin{scope}
				\clip (2,\start) rectangle (-2,\ende);
				\draw[thick, cyan, dashed] (\x) circle (1.5);
			\end{scope}
		}
		
				\foreach \x in {a8,a9,a2,a1,a0,b1,b2,b8,b9}{
			\fill[black] (\x) circle (0.1);
		}		\foreach \x in {q05,q06,q5,q8,q9}{
			\fill[black] (\x) circle (0.1);
		}
		
				\draw[black] (a9) -- (a8);
		\draw[black] (b9) -- (b8);
		\draw[black] (a2) -- (a1) -- (a0) -- (b1) -- (b2);
		\draw[black,dotted] (a8) -- (a2) (b1) -- (b8);
				\draw[black] (a0) -- (q9) -- (q8);
		\draw[black] (q06) -- (q05);
		\draw[black,dotted] (q8) -- (q5) (q5) -- (q05);
		
				\draw[densely dashed, blue] (-4,6)
			to[out=270, in=180] (0,4.6)
			to[out=0, in=270] (3,6)
			to (3,15)
			to[out=90, in=0] (0,16.4)
			to[out=180, in=90] (-4,15)
			to (-4,6);
			
				\begin{scope}
			\clip (-4.1,16.5) rectangle (-1.5,4.5);
			\path[pattern=north east lines, pattern color=orange] (-4,6)
				to[out=270, in=180] (0,4.6)
				to[out=0, in=270] (3,6)
				to (3,15)
				to[out=90, in=0] (0,16.4)
				to[out=180, in=90] (-4,15)
				to (-4,6);
		\end{scope}
		
				\node at (-9,2) {$G_P$};
		\node at (-9,3.1) {$G'$};
		\node at (-9,4.3) {$G''$};
		\node at (2.2,3) {$G_Q$};
		\node at (3.5,15) {$C$};
		\node at (-3,11.5) {$C'$};
        \node at (0.7, 5) {$q_m$};
        \node at (0.7, 16) {$q_0$};
        \node at (0, -0.5) {$p_t = q_\ell$};

		\end{tikzpicture}
		\caption{The setting of the proof of \cref{starwidthclaim5}.}
		\label{fig:star-claim5}
	\end{figure}
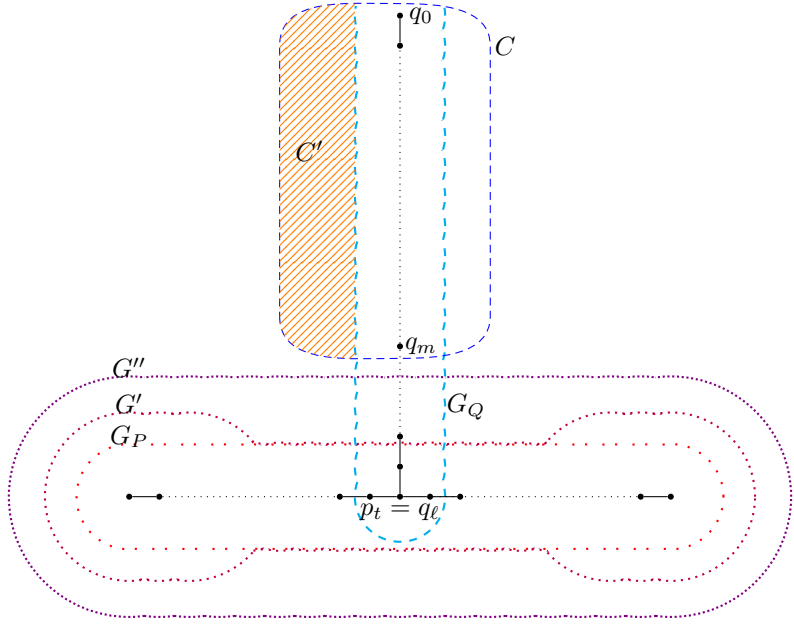

	\begin{claimproof}
        The reader may look at \cref{fig:star-claim5} to follow the proof more easily.
		Let $q_0$ be a vertex of $C$ of maximal distance from $P$.
                Since $C \in \cC$ the vertex $q_0$ has distance at least $24k + 5$ from $P$.
		Let $Q = q_0 \ldots q_{\ell}$ be a shortest path in $G$ from $q_0$ to $P$ (so $Q$ is geodesic in~$G$).
        Note that $C$ is contained in a component of $G - G'$, and thus $Q$ ends in a vertex $p_t=q_{\ell}$ with $s - 12k - 3 \leq t \leq s + 12k + 3$ by \cref{starwidthclaim2}.
        In particular $3k < t < n - 3k$.
        Let $m$ be the last index such that $q_m \notin G''$.
        Then $Qq_m$ is a shortest $q_0$--$N_G(G'')$-path, and $m = \ell - 18k - 4$.
        
        For every $0 \leq i \leq \ell$ we define $B_i^Q$ to be the ball in $G$ of radius~$3k$ around~$q_i$, and let $G_Q = G[\bigcup_{i = 0}^\ell B_i^Q]$.
        We claim that every vertex of $C - G_Q$ has distance at most $12k$ from $q_0$. 
        To see this, let~$C'$ be any component of $C - G_Q$, and let~$C''$ be the component of $G - G_Q$ containing~$C'$.
		We first show that $C''$ has no neighbour in $\bigcup_{i = 3k+1}^\ell B_i^Q$ and that this in particular implies that $C' = C''$.

		Towards a contradiction, we first assume that $C''$ has a neighbour in some~$B_j^Q$ with \mbox{$3k < j < \ell - 12k - 3$}.
        Then $q_j Q$ has length at least $12k + 4 = 4\cdot(3k + 1)$.
        So \cref{lem:FindingA3geodesicW}~\cref{CaseQ} applied to $p_{t - 3k - 1}P p_{t + 3k + 1}$, $q_{j - 3k - 1}Q$ and a shortest path from $C''$ to $q_j$ yields that~$G$ contains a $3$-quasi-geodesic $\subdivk{W}{3k}$, a contradiction.

        Second, suppose for a contradiction that~$C''$ has a neighbour in $\bigcup_{i = \ell - 12k - 3}^\ell B_i^Q$. 
        In addition,~$C''$ has a neighbour in $\bigcup_{i = 0}^{3k} B_i^Q$. 
        Indeed, since vertices in $\bigcup_{i = \ell - 12k - 3}^\ell B_i^Q$ have distance at most $15k + 3 < 18k+4$ from~$P$, they are neither contained in~$C'$ nor do they have neighbours in $C'$. 
        Since $C' \subseteq C''$, this implies that~$C'$ has no neighbours in $\bigcup_{i = 3k+1}^\ell B_i^Q$.
		As $C$ is connected and contains $q_0 \in G_Q$, $C'$ has a neighbour in~$G_Q$ which then has to be contained in $\bigcup_{i = 0}^{3k} B_i^Q$.
        Hence, $C''$ also has a neighbour in $\bigcup_{i = 0}^{3k} B_i^Q$.
		
        For $r = 3k$, $m_0 = 3k$ and $m_1 = 12k + 3$, we have that
        \[
        \ell - m_0 - m_1 - 2r = \ell - 3k - 12k - 3 - 6k = \ell - 21k - 3 \geq 2k + 2,
        \]
		so we can apply \cref{lem:FindingALongGeodesicCycle} to $Q$ and $C''$ to obtain a geodesic cycle of length at least $2(2k + 2) > 3k + 3$, which contradicts our assumption on $G$.
        Thus, $N_G(C'') \subseteq \bigcup_{i = 0}^{3k} B_i^Q$.
        As $G''$ is connected and contains $B^Q_\ell$, it follows that $C'' \cap G'' = \emptyset$ and thus $C' = C''$, if $G''$ and $\bigcup_{i = 0}^{3k} B_i^Q$ do not meet.
        
        Indeed, suppose for a contradiction that $G''$ and $\bigcup_{i = 0}^{3k} B_i^Q$ meet in a vertex $v$. 
        Then 
        \[d_G(P,q_0) \leq d_G(P,v) + d_G(v,q_i) + d_G(q_i,q_0) \leq 18k+4 +3k +3k = 24k+4,\]
        which contradicts the choice of $q_0$.
		
		We can now show that all vertices in $C'$ have distance at most $12k$ from $q_0$. For this, consider the induced subgraph of $G$ on the vertex set $V(G'') \cup V(C)$.
		In this graph, $q_0$ has maximal distance from $P$ and thus we can apply \cref{lem:ComponentsAttachingToTheEndOfP} to the shortest $q_0 - P$ path $Q$, $r = 3k$, $\ell = 3k$ and the component~$C'$.
		So every vertex in $C'$ has distance at most $9k$ from $Q$.
		As every shortest $C'$--$Q$ path ends in some $q_i$ with~$i\leq 3k$, this implies that every vertex in $C'$ has distance at most $12k$ from $q_0$.

		Now we obtain a path-decomposition of $G[V(C) \cup V_s]$ as follows.
		In a first step, for $i\leq \ell$ we define
		\[
		U_i' := \bigcup_{q_j \in B_{Q}(q_i, 3k+1)} B_j^Q.
		\]
        By \cref{lem:UnionOfBallIsDecomp}, the $U_i'$ are the bags of a $Q$-decomposition of $G_Q$ of radial spread at most $4\cdot (3k) + 2 = 12k+2$; and every vertex in any $U_i'$ has distance at most $6k + 1$ from $q_i$ in $G[U'_i]$.
        
		In a second step, for $Q^* = Qq_m$, we define a $Q^*$-decomposition of $G_Q^* \cup C$ where $G_Q^* := G[\bigcup_{i = 0}^{m+3k+1} B_i^Q]$.
        For that, we let $U''_i := U_i'$ for $1\leq i \leq m$.
		The neighbourhood of every component of $C - G_Q$ is contained in $U_0'$, so we can just add its vertices to $U_0'$: let $U''_0$ be the union of $U_0'$ and all vertices of components of $C - G_Q$.
		By construction $U''_0$ has radius at most $12k$.
        Then the $U''_i$ are the bags of a $Q^*$-decomposition of $G_Q^* \cup C$. Indeed, every vertex of $C - G_Q$ is contained in $U''_0$ by definition, and every vertex of $C \cap G_Q$ is contained in $G_Q^*$ since all vertices in $\bigcup_{i = m + 3k + 2}^\ell B_i^Q$ have distance at most $18k + 3$ from $P$ and are thus not contained in $C$. Moreover, $(Q^*, \cU'')$ has radial width at most $12k$ and radial spread at most $12k+2$.

        We now restrict the bags~$U''_i$ to~$C$ and add $V_s$ to all $U''_i$ with $i \geq m - 6k$, i.e.\ we set $U_i := (U''_i \cap V(C))$ for $i < m-6k$ and $U_i := (U''_i \cap V(C)) \cup V_s$ for all $i \geq m-6k$. To verify that $(Q^*, \cU)$ is a decomposition of $G[V(C) \cup V_s]$, it suffices to show that $N_G(G-C) \subseteq U_m$.

                For this, let $v \in N_G(G-C) \subseteq V(C)$ be given, and let~$w$ be a neighbour of~$v$ in $G-C$. 
        As~$C$ is a component of $G - G''$, this means that $w \in G''$. 
        As we have already seen above, no component of $C - G_Q$ has neighbours in~$G''$, so $v \in G_Q$, which in particular implies that~$v \in G_Q^*$. 
        If $v \in \bigcup_{i = m-3k-1}^{m + 3k+1} B_i^Q$, then $v\in U_m$ by construction. 
        So we may assume that $v \in B_j^Q$ for some $j < m - 3k - 1$.
        But this implies that $q_0$ has distance at most 
        \[
        \dist_G(q_0, q_j) + \dist_G(q_j, v) + \dist_G(v, w) + \dist_G(w, P) \leq (m - 3k - 2) + 3k + 1 + (18k+4) = m + 18k + 3 < \ell
        \]
        from $P$, which contradicts that~$Q$ is a shortest $q_0$--$P$ path.

        Hence, $(Q^*, \cU)$ is a path-decomposition of $G[V(C) \cup V_s]$, and it has radial spread at most $12k+2$ by construction. Moreover, it has radial width at most $54k+8$: Let $i \leq m$ and let $v \in U_i$. If $i \leq m-6k$, then, since $Q^* \subseteq C$ is a shortest $q_0$--$N_G(G-C)$ path in $G$ and $G[U'_i]$ has radius at most $12k$, we have $U_i = U'_i$ and hence $G[U_i]$ has radius at most $12k$. Now suppose that $i > m-6k$. If $v \in V_s$, then $v$ has distance at most $36k+7$ from $p_s$ by \cref{claim:Starwidth:New}. Otherwise, if $v \in U_i \setminus V_s$ then
        \[
        d_{G[U_i]}(v, p_s) < d_{G[U_i]}(v, q_i) + d_{G[U_i]}(q_i, q_m) + d_{G[U_i]}(q_m, p_s) \leq 12k + 6k + (36k+8) = 54k + 8,
        \]
        where we used that $q_iQq_m \subseteq G[U_i]$ by definition, that $q_m \in N_G(V'_s) \subseteq V_s$ by \cref{starwidthclaim4}, and that every vertex in $V_s$ has distance at most $36k+7$ from $p_s$ by \cref{claim:Starwidth:New}.
	\end{claimproof}

    We now combine these path-decompositions of the components of $G - V'''$ with the path-decomposition $(P, \cV)$ of~$G'''$ to a star-decomposition of~$G$.
	For this, recall that by \cref{claim:Starwidth:New}, $(P, \cV)$ is a path-decomposition of $G'''$ of radial width at most $36k + 7$ and radial spread at most $24k+5$. Moreover, the components in $\cC$ are precisely the components of $G-V'''$.
	For every component $C\in \cC$, \cref{starwidthclaim5} guarantees the existence of a decomposition $(Q_C,\cU_C)$ of $C$ modelled on a path $Q_C$ of radial width at most $54k+8$ and radial spread at most $12k+2$ such that the bag in $\cU_C$ corresponding to the last vertex of $Q_C$ contains $V_s$.
    
    We obtain a subdivided star $S$ from the disjoint union of $P$ and the $Q_C$ by adding edges from the last vertices of the $Q_C$ to $p_s$.
	Now every vertex $h$ of $S$ already has a bag, which we denote by $V_h$, in exactly one of the path-decompositions $(P, \cV)$ or $(Q_C, \cU_C)$.
	It is straightforward to check that $(S, \cV)$ is a star-decomposition of~$G$ and that is radial width is at most $54k+8$. Moreover, its radial spread is at most $(24k+5) + 1 + 6k = 30k+7$, as only vertices in $V_s$ may lie in more than one decomposition of the form $(P, \cV)$ or $(Q_C, \cV)$. This completes the proof.
    \end{proof}

We remark that the proof actually yields that the subdivided star $S$ is $(1, 60k+14)$-quasi-isometric to $G$ (where we may choose $\phi(h) = h$ for all vertices $h$ of $S$ as $V(S) \subseteq V(G)$ by construction). Here is a hint for the proof: The proof of \cref{thm:starWidth} shows that all vertices in $G$ have distance at most $58k+9 \leq 60k+14$ from some vertex of $S$. Moreover, since $P$ and all paths $Q_C$ are geodesic in $G$, it remains to check the distances of vertices $h, h'$ of $S$ that lie on distinct such paths. Assume $h \in V(Q_C)$ and $h' \in V(Q_{C'})$ ($h \in V(Q_C)$ and $h' \in V(P)$ is similar). Then $d_S(h, h') \leq d_G(h, h')$ as $Q_C \subseteq C$ is a shortest path in $G$ between its first vertex and $N_G(G-C)$. Also, $d_G(h, h') \leq d_S(h,p_s) + d_G(N_G(G-C), N_G(G-C')) + d_S(p_s, h') = d_S(h, h') + 60k + 14$ since $p_s$ is the centre of the subdivided star $S$ and because $d_G(N_G(G-C), N_G(G-C')) \leq 2\cdot(18k+4) + 2\cdot(12k+3)$ follows from the fact that $C, C'$ are components of $G-V'' = G - B_G(P, 18k+4)$ and from \cref{starwidthclaim2}.


\arXivOrNot{
\appendix

\section{Where quasi-geodesic topological minors are not the right obstruction} \label{app:CterexQuestionRadWidth}

Davies, Hickingbotham, Illingworth and McCarty showed that~\cref{conj:GeorgakopoulosPapasoglu} is false in general.
Hence, the strengthening of~\cref{conj:GeorgakopoulosPapasoglu} to quasi-geodesic topological minors also has a negative answer.
However, this strengthening fails earlier than~\cref{conj:GeorgakopoulosPapasoglu}, and in this appendix we exhibit sets~$\cX$ exemplifying this.
To make this precise, let us give some background on minor-closed classes of graphs.

A class~$\cG$ of graphs is~\emph{minor-closed} if~$G' \in \cG$ whenever~$G'$ is a minor of some~$G \in \cG$.
Further, a class~$\cG$ of graphs is \emph{defined by a finite list~$\cX = \{X_1, \dots, X_\ell\}$ of forbidden (topological) minors} if a graph~$G$ is contained in~$\cG$ precisely if no~$X_i$ is a (topological) minor of~$G$.

The Graph Minor Theorem by Robertson and Seymour characterises minor-closed graph classes.
\begin{theorem}[\cite{GMXX}]\label{thm:MinorClosedForbiddenListMinors}
    Every minor-closed class of graphs is defined by a finite list of forbidden minors.
\end{theorem}

Robertson and Seymour showed in~\cite{GMVIII}*{(2.1)} that every minor-closed graph property that is defined by a finite list of forbidden minors is also defined by a still finite, but usually significantly larger list of forbidden topological minors.
This then yields the following corollary of~\cref{thm:MinorClosedForbiddenListMinors}:

\begin{corollary} \label{lem:MinorClosedForbiddenListTopMinors}
    Every minor-closed class of graphs is defined by a finite list of forbidden topological minors. \qed
\end{corollary}

Consider a minor-closed class of graphs defined by a finite list~$\{M_1, \dots, M_k\}$ of forbidden minors~(\cref{thm:MinorClosedForbiddenListMinors}) and by a finite list~$\{X_1, \dots, X_\ell\}$ of forbidden topological minors~(\cref{lem:MinorClosedForbiddenListTopMinors}).
How are the forbidden minors~$M_i$ and the forbidden topological minors~$X_j$ related to each other?
Robertson and Seymour described their structure in~\cite{GMVIII}*{Section~2} by `splitting' a vertex. Equivalently, one may reformulate it as follows:
the graphs~$X_1, \dots, X_\ell$ are the models of the~$M_1, \dots, M_k$ such that all the induced subgraphs on the branch sets are trees and if the branch set has size at least two, then every vertex in the branch set has degree at least~$3$ in~$G$.
In particular, each~$X_i$ \emph{arises} from some~$M_j$ in that way.

With this at hand, the following question indeed generalises the strengthening of~\cref{conj:GeorgakopoulosPapasoglu} for quasi-geodesic topological minors:
\begin{question} \label{qu:ConjStrengtheningTopMinors}
     Let~$\cM$ be a finite set of finite graphs, and let $\cX$ be the finite set of graphs which defines the class of graphs with not $M$ minor with $M \in \cM$ via forbidden topological minors; see \cref{lem:MinorClosedForbiddenListTopMinors}.
     Does there exists a function $f' : \N \rightarrow \N \times \N$ such that every graph with no $(\ge ck)$-subdivision of some $X \in \cX$ as a $c$-quasi-geodesic subgraph for some~$c \in N$ is $f'(k)$-quasi-isometric to a graph with no topological $X$ minor for any $X \in \cX$?
\end{question}

As stated before, \cref{qu:ConjStrengtheningTopMinors} has a negative answer even when~$\cX$ is the corresponding list of forbidden topological minors to very simple sets~$\cM$ such as~$\{K_{3,3}, K^5\}$, which characterises the planar graphs~\cite{wagner1937eigenschaft}.
The following proposition provides a sufficient criterion for this, via the detour of graph-decompositions as established in~\cref{prop:QuasiIsoIntro}.

\begin{proposition} \label{prop:QuObstrNegativeAnswer}
    Let~$\cM$ be a finite set of finite graphs, let~$\cH$ be the class of graphs without $M$ minor for~$M \in \cM$, and let~$\cX$ be the finite list of forbidden topological minors characterizing~$\cH$ (which exists by~\cref{lem:MinorClosedForbiddenListTopMinors}).
    Let~$m$ be the smallest number of vertices of degree~$\ge 3$ among all~$M \in \cM$.
    Suppose that for each~$M \in \cM$, if~$M$ has precisely~$m$ vertices of degree at least~$3$, then it also has a vertex of degree~$4$.
    
    Then there exists, for every given~$R \in \N$, a graph~$G$
    which has radial $\cH$-width at least~$R$ and contains no~$(\geq 1)$-subdivision of~$X_i$ for all~$1 \le i \le \ell$.

                \end{proposition}

\begin{proof}
    Let~$\cM = \{M_1, \dots, M_k\}$.
    Without loss of generality, we may assume that~$M_1$ has precisely~$m$ vertices of degree at least~$3$, has minimal maximum degree among all such~$M_j$, and is minimal among these~$M_j$ with respect to the number of vertices of degree~$\Delta(M_1)$.
    Let~$G'$ be the graph obtained from~$M_1$ by subdividing each edge at least~$4 R$ times.
    We then let~$G$ be the graph that arises from~$G'$ by `splitting' a vertex~$v$ of degree~$\Delta(M_1)$ into two adjacent vertices, one of which has degree exactly~$3$.
    More formally, $G$ is the model of~$G'$ that differs from~$G'$ only in that the branch set to~$v$ consists of the endvertices~$u_e$ and~$v_e$ of an edge~$e \in G'$ such that~$u_e$ has degree~$3$ and thus~$v_e$ has degree~$d_{G'}(v) - 1$.
    Note that~$G$ has precisely one more vertex of degree at least~$3$ than~$M_1$, maximum degree at most~$\Delta(M_1)$, and one vertex of degree~$\Delta(M_1)$ less than~$M_1$.
    We now show that~$G$ is as desired.

    Suppose for a contradiction that~$G$ has an~$\cH$-decomposition of radial width less than~$R$.
    From it, we then obtain an $\cH$-decomposition of $G'$ by adding $v$ to every bag that contains at least one of $v_e$ and $u_e$ and deleting $v_e$ and $u_e$ from all bags.
    By the construction of~$G$ from~$G'$, this~$\cH$-decomposition of~$G'$ has again radial width less than~$R$.
    But~$G'$ has \radialHWidth{$\cH$} at least $R$ by~\cref{lem:quasigeodesictopologicalminorsareobstructions}, a contradiction.
    
    It remains to show that~$G$ contains no~$(\geq 1)$-subdivision of any~$X \in \cX$. 
    Suppose for a contradiction that~$G$ contains a subgraph~$X_G$ that is a~$(\geq 1)$-subdivision of some~$X \in \cX$.
    This~$X$ arises from some~$M \in \cM$.
    We now distinguish two cases depending on whether~$M$ has precisely~$m$ vertices of degree at least~$3$ or not.

    If~$M$ has more than~$m$ vertices of degree at least~$3$, then~$M$ and hence~$X$ each have at least~$m+1$ vertices of degree at least~$3$.
    Thus, all vertices of degree at least~$3$ in~$G$ are branch vertices of~$X_G$, and in particular, $X_G$ contains both~$u_e$ and~$v_e$ as branch vertices.
    Since~$d_G(u_e) = 3$, this yields~$e \in X_G$.
    So~$X_G$ is not a~$(\geq 1)$-subdivision of~$X$, contradicting our assumption.

    So suppose that~$M$ has precisely~$m$ vertices of degree at least~$3$.
    With the same argument as above, $u_e$ and~$v_e$ cannot both be branch vertices of~$X_G$ of degree at least~$3$.
    Thus, $X_G$ has at most~$m$ branch vertices of degree at least~$3$.
    This implies~$X = M$, as~$X$ arises from~$M$.
    Now we cannot have~$\Delta(M) > \Delta(M_1)$, as~$\Delta(G) \le \Delta(M)$ and the subdivision~$X_G$ of~$X = M$ is a subgraph of~$G$.
    So by our minimal choice of~$M_1$, $M$ has maximum degree~$\Delta(M_1)$ and contains at least as many vertices of degree~$\Delta(M_1)$ as~$M_1$.
    But~$G$ has one vertex of degree~$\Delta(M_1)$ less than~$M_1$, and this contradicts that the subdivision~$X_G$ of~$X = M$ is a subgraph of~$G$.
    This completes the proof.
\end{proof}

\section{Counterexamples to stronger versions in the star and tree case} \label{app:Outlook}

A possible strengthening of \cref{thm:star-width-intro} and~\cref{conj:tree-width-intro} is suggested by the proofs of the path- and cycle-case (see \cref{thm:RadialPathWith:Technical,thm:RadialCycleWidth:Technical}).
The path- and cycle-decompositions which we exhibit there have an additional property: their decomposition graph~$H$ is itself a $C$-quasi-geodesic (in fact geodesic) subgraph of~$G$ such that every vertex of~$H$ is contained in its corresponding bag; here, the constant~$C$ is independent of the constant~$c$ in `$c$-quasi-geodesic~$(\geq ck)$-subdivision' in \cref{thm:star-width-intro} and~\cref{conj:tree-width-intro}.
The star-decomposition constructed in the proof of~\cref{thm:star-width-intro}, however, does not have this extra property, and in fact there is not always such a strengthened star-decomposition as the following example shows (see \cref{example:appendixstar} for more details).

\begin{example} \label{ex:DecompStarNotGeodesic}
    Let $d, k \geq 1$ be integers.
    Let $G$ be a subdivision of the graph $W$ depicted in \cref{fig:wrench} such that the edge between the two vertices of degree $3$ is subdivided less than $k$ times and every other edge is subdivided at least $d$ times.
    Then $G$ has radial star-width at most $\lceil k/2 \rceil$.
    But if a subdivided star $S$ is a subgraph of $G$, then any $S$-decomposition of $G$ with the extra property that every vertex of $S$ is contained in its corresponding part has radial width at least $\lceil (d+1)/2 \rceil$.
\end{example}

Still, the proof of \cref{thm:starWidth} can be adapted (with a slight worsening of the bounds) to obtain the following weaker property for star-decompositions:
For a large, but bounded integer~$C$, the decomposition star $S$ can be obtained from a tree~$T$ that is a $C$-quasi-geodesic subgraph of $G$ by contracting a subtree of bounded radius to a single vertex~$s$.
This vertex $s$ is then the centre of the star $S$, and the set of vertices of $T$ that were contracted to $s$ is a subset of the bag of the decomposition that corresponds to $s$.

An analogue of~\cref{ex:DecompStarNotGeodesic} also holds in the more general case of tree-decompositions:
there is no~$C \in \N$ such that if a graph~$G$ has low \radialHWidth{tree}, then there exists a tree-decomposition of~$G$ of small {\radialWidth} whose decomposition tree is a~$C$-quasi-geodesic subgraph of~$G$.
We prove this with our rather involved \cref{ex:TreeOfWheels}, which is related to constructions in~\cites{KLMPW, bergerseymourboundeddiamTD}.
The existence of such an example suggests that any further work towards an answer to~\cref{conj:tree-width-intro} will have to build on new methods for constructing graph-decompositions of low radial width.

\subsection{Counterexamples to stronger structural dualities}\label{appendix:explanation}

Here we present further details for the examples mentioned above.
We first give examples of graphs of large \radialHWidth{path} which contain neither a long~$c$-quasi-geodesic cycle nor a~$c$-quasi-geodesic~$(\geq 1)$-subdivision of the claw for any~$1 \le c < 2$. 

\begin{example}\label{example:appendixgeodesic}
    Let $k$ be an integer.
    \begin{enumerate}
        \item Assume $k \geq 1$. By completing the neighbourhood of the centre vertex of the $k$-subdivided claw to a $K^3$, we obtain a graph $G$ which contains a $2$-quasi-geodesic $k$-subdivided claw but $G$ contains neither a $c$-quasi-geodesic cycle of length $\geq 4$ nor a $c$-quasi-geodesic claw for every $c$ with $1 \leq c < 2$.
        \item (\cite{bergerseymourboundeddiamTD}) Assume $k \geq 2$. The $k \times k$ grid has radial tree-width at least $(k-2)/2$ by \cref{lem:quasigeodesictopologicalminorsareobstructions}, since it contains a cycle of length $4(k-1)$ as $2$-quasi-geodesic subgraph; but every $c$-quasi-geodesic cycle with $1 \leq c < 2$ has length $4$.
    \end{enumerate}   
\end{example}

Our second example, a rephrasing of~\cref{ex:DecompStarNotGeodesic}, demonstrates that, given a graph~$G$ of small \radialHWidth{star}, we cannot always find a star-decomposition whose decomposition star is a~$c$-quasi-geodesic subgraph of~$G$ such that every vertex of this subgraph belongs to its corresponding bag.

\begin{example}\label{example:appendixstar}
    Let $k,d \geq 0$ be integers.
    Let $G$ be a subdivision of the graph $W$ depicted in \cref{fig:wrench} such that the edge of $W$ between the two vertices $u,v$ of degree $3$ is subdivided at most $k$-times and every other edge is subdivided at least $d$-times.
    Then $G$ has a star-decomposition $(S,(V_t)_{t \in S})$ of radial width at most~$\lceil (k+1)/2 \rceil$, but there is no star-decomposition $(S',(V'_t)_{t \in S})$ of radial width at most $\lfloor d/2\rfloor $ such that the star $S'$ is a subgraph of $G$.
\end{example}

\begin{proof}
    First, we show that there is a star-decomposition $(S,(V_t)_{t \in S})$ of low radial width: 
    Let $S$ be the subdivided star obtained from $G$ by contracting the unique $u$--$v$ path $P$ in $G$ to a vertex $s$.
    Now let $V_s := V(P)$, and $V_t := \{t\} \cup N_G(t)$ for every other node $t \neq s$ of $S$.
    Then $(S,(V_t)_{t \in S})$ is the desired star-decomposition of $G$.

    Secondly, we show that every star-decomposition $(S',(V'_t)_{t \in S})$ of $G$ where $S'$ is a subgraph of $G$ has radial width at least $\lfloor d/2 \rfloor$:
    Since $S'$ is a subgraph of $G$ which is a star, it is either a path or a star centred at $u$ or $v$.
    Hence, one of the four vertices of $G$ of degree $1$ has distance $d+1$ to $S'$. 
\end{proof}

We now present an example that in general, even for graphs with low radial tree-width, we cannot hope to obtain tree-decompositions of low radial width where the decomposition tree is a $c$-quasi-geodesic subtree (with bounded $c$).
The trees of wheels form a class of graphs that all have radial tree-width~$1$.
But by \cref{lem:NocGeodesicDecompositionTree}, for all $c,n\in \mathbb{N}$ there is a tree of wheels $G$ such that every $H$-decomposition of $G$, with a~$c$-quasi-geodesic subtree $H$ of $G$ with $h\in V_h$ for all nodes $h\in H$, must necessarily have radial width at least~$n$.

\tikzset{spokes3/.pic = {    \node [vtx] at (0,0) {};
    \draw (0,0) -- (90:1cm) (0,0) -- (210:1cm) (0,0) -- (330:1cm);
}}

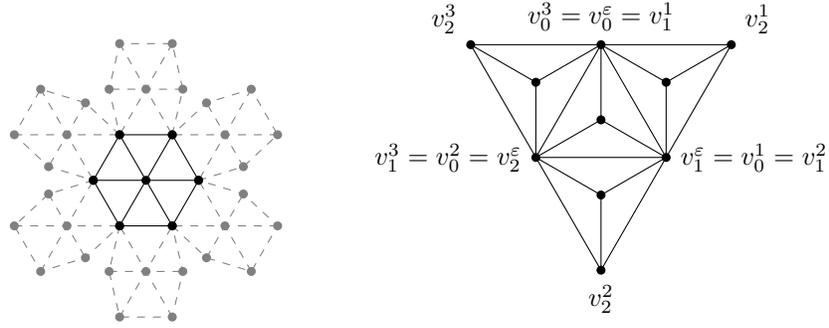
\begin{figure}
    \centering
    \begin{tikzpicture}[scale=0.7]         \tikzset{wheel6/.pic = {            \draw (0,0) node [vtx] {};
            \draw (0:.7cm) -- (180:.7cm)
            (60:1cm) -- (240:1cm)
            (120:1cm) -- (300:1cm);
            \draw (300:1cm) node [vtx] {}
            -- (0:.7cm) node [vtx] {}
            -- (60:1cm) node [vtx] {}
            -- (120:1cm) node [vtx] {}
            -- (180:.7cm) node [vtx] {}
            -- (240:1cm) node [vtx] {};
        }}
        \node [vtx] at (0,0) {};
        \foreach \a in {0,60,120,180,240,300}
            \draw (\a:1cm) ++ (\a + 60:1cm) pic [scale = 0.7, rotate = \a - 60, dashed, gray] {wheel6};
        \foreach \a in {0, 60, 120, 180, 240, 300}
           \draw (0,0) -- (\a:1cm) node [vtx] {} -- (\a + 60:1cm);
    \end{tikzpicture}
    \hskip 1cm
    \begin{tikzpicture}         \draw (0,0) pic {spokes3};
        \draw (90:1cm) node [vtx, label = {$v^3_0 = v^{\varepsilon}_0 = v^1_1$}] {}
        -- (210:1cm) node [vtx, label = {180:$v^3_1 = v^2_0 = v^{\varepsilon}_2$}] {} 
        -- (330:1cm) node [vtx, label = {0:$v^{\varepsilon}_1 = v^1_0 = v^2_1$}] {}
        -- cycle;
        \foreach \a in {30, 150, 270}
           \draw (\a:1cm) pic [rotate = 60] {spokes3};
        \draw (30:2cm) node [vtx, label = 30:$v^1_2$] {} 
        -- (150:2cm) node [vtx, label = 150:$v^3_2$] {} 
        -- (270:2cm) node [vtx, label = 270:$v^2_2$] {} 
        -- cycle;
    \end{tikzpicture}
    \caption{Two trees of wheels.
    The left graph is a tree of wheels of depth $1$ for $n=6$, and the edges that do not belong to the central wheel are dashed.
    The right graph is a tree of wheels of depth $1$ for $n=3$, and the vertices that are not hubs are labelled.}
    \label{fig:TreeOfWheels}
\end{figure}

\begin{example}[Tree of wheels]\label{ex:TreeOfWheels}
	Let $n\geq 3$ be an integer.
	Recall that a \emph{wheel $W_n$} is the graph obtained from a cycle $C^n$ and a single vertex $K^1$ by adding edges from all vertices in the cycle to $K^1$.
	We will denote the cycle, the \emph{rim} of the wheel, by $v_0v_1\cdots v_n$ where $v_0=v_n$, and the edge $v_iv_{i+1}$ by $e_{i+1}$.
	The vertex of~$K^1$ is the \emph{hub} of the wheel.
	
	In the following construction, all wheels are copies of $W_n$.
	We start with a wheel in \emph{depth} $0$, which we refer to as the \emph{central wheel}.
	In the first step, we add for every edge on the rim of the central wheel another wheel.
	More precisely, we take, for each edge~$e_i$ on the rim of the central wheel, a new wheel and identify the vertex~$v_1$ of the new wheel with $v_{i-1}$ of the central wheel, identify $v_0$ of the new wheel with $v_{i}$ of the central wheel and identify the edge $e_1$ of the new wheel with $e_i$ of the central wheel.
	We consider the new wheels to be in \emph{depth} $1$.
	
	We now repeat this step: 
	Say we have already added wheels in depth $j-1$ but not wheels in depth $j$.
	At each edge~$e_i$ with $i\neq 1$ of one of the wheels in depth $j-1$ we add another wheel.
    We consider the new wheel to be in \emph{depth} $j$.
	More precisely, let $e_i$ with $i\neq 1$ be an edge of a wheel~$W$ in depth $j-1$.
	We then add a new wheel for $e_i$, identify $v_1$ of the new wheel with $v_{i}$ of~$W$, identify $v_0$ of the new wheel with $v_{i+1}$ of~$W$ and identify $e_1$ of the new wheel with~$e_i$.
	
	The graph which we obtained in step~$d$ (that is, we added the wheels in depth $d$ but no wheels in depth~$d+1$) is the \emph{tree of wheels of depth~$d$}.
\end{example}

Two trees of wheels of depth $1$ are depicted in \cref{fig:TreeOfWheels}.
Note that for every tree of wheels, there is a tree-decomposition where every bag is the vertex set of a wheel.
So every tree of wheels has radial tree-width~$1$.

For the following lemmas we need to establish some notation regarding trees of wheels.
Let~$G$ be a tree of wheels of depth~$r$.
Every wheel in depth $j$ was added at an edge~$e_{i}$ of a wheel in depth $j-1$, and this latter wheel was again added at an edge $e_{i'}$ of a wheel in depth $j-2$ to the central wheel and so on.
So every wheel is uniquely identified by the indices of the edges to which previous wheels were added.
More precisely, we associate the central wheel with the empty sequence~$\varepsilon$.
When a wheel is added at edge~$e_i$ of a wheel that is associated with the sequence $i_1 i_2 \cdots i_j$, then the new wheel is associated with the sequence~$i_1 i_2 \cdots i_j i$.
For a sequence~$s$ that is associated with a wheel, we also denote the edge~$e_i$ of that wheel by $e^s_i$ and the vertex~$v_i$ of that wheel by~$v^s_i$.

In~\cite{KLMPW}*{Section~6} a graph is described that is very close to the tree of wheels on $3 + 1$ vertices, only without the hubs. 
There it is also proven that for a given $c$, such a graph of sufficiently large radius does not have a spanning tree that is $c$-quasi-geodesic.

\begin{lemma}\label{lem:NocGeodesicSpanningTree}
	Let\footnote{The proof can be adapted to $n=3$, but then we get a different bound on~$c$.} $n\geq 4$, $r,c\in \mathbb{N}$ and let $G$ be a tree of wheels on $n+1$ vertices of depth~$d$.
	If $G$ contains a $c$-quasi-geodesic tree that contains all vertices on rims of wheels of $G$, then $c\geq 2(d+1)$.
\end{lemma}

\begin{proof}
	We need some additional notation.
	Let $s\neq \varepsilon$ be a sequence associated with a wheel.
    Then $G$ splits into two connected graphs $G^s_1$ and $G^s_2$ which only meet in $G[\{v^s_0, v^s_1 \}]$.
    We choose $G^s_1$ such that it contains the hub of the central wheel.
	
	Assume that $G$ contains a $c$-quasi-geodesic tree~$T$ that contains all vertices on rims of wheels of $G$.
	We will now recursively construct a sequence~$s$ such that the wheel associated with $s$ contains two vertices whose distance in~$T$ is large.
	In the first step let $I \subseteq \{1, \ldots, n\}$ consist of those indices $i$ such that $v_i = v^i_1$ and $v_{i+1} = v^i_0$ are contained in the same component of $T \cap G^i_2$.
	As $T$ is acyclic and the $G^i_2$ are edge-disjoint, there is $i\in \{1,\ldots,n\}$ with $i\notin I$.
	In particular, $e^i_0$ is not contained in~$T$, and $v^i_0$ and $v^i_1$ are contained in the same component of $T \cap G^i_1$, since $T$ is connected and all vertices on rims of wheels are in $T$ by assumption. 
	Thus, the unique $v^i_0$--$v^i_1$ path in~$T$ is contained in $T \cap G^i_1$ and has at least two edges.
	
	Now we repeat the following step recursively.
	Assume that $j\leq d-1$ and that a sequence $s=i_1 i_2 \ldots i_{j}$ associated with a wheel has been constructed such that $v^s_0$ and $v^s_1$ are not contained in the same component of $T \cap G^s_2$ and have distance at least $2j$ in~$T$.
	Then it directly follows that there also is some $i\neq 1$ such that $v^s_i = v^{si}_1$ and $v^s_{i+1} = v^{si}_0$ are not contained in the same component of $T\cap G^s_2$.
	As $G^{si}_2$ is a subgraph of~$G^s_2$, the vertices $v^{si}_0$ and $v^{si}_1$ are not contained in the same component of $T\cap G^{si}_2$.
	Furthermore, the unique $v^{si}_0$--$v^{si}_1$~path in~$T$ contains $v^s_0$ and $v^s_1$ and thus, since $n \geq 4$, $v^{si}_0$ and $v^{si}_1$ have distance at least~$2(j+1)$ in~$T$.
	
	So assume we obtained a sequence $s=i_1 i_2 \ldots i_r$ such that $v^s_0$ and $v^s_1$ are not contained in the same component of $T\cap G^s_2$ and have distance at least $2d$ in~$T$.
	As above, there is some integer $i\neq 1$ such that $v^s_i$ and $v^s_{i+1}$ are not contained in the same component of $T\cap G^s_2$, and thus they have distance at least~$2(d+1)$ in $T$ while they are adjacent in $G$, implying $c\geq 2(d+1)$.
\end{proof}

\begin{lemma}\label{lem:NocGeodesicDecompositionTree}
	Let\footnote{The proof can be adapted to $n\geq 3$, but then we get a different bound on~$c$.} $n\geq 6$, $r\in \mathbb{N}$ and let $G$ be a tree of wheels on $n+1$ vertices of depth~$d$.
	If there is a tree~$T$ in~$G$ that is $c$-quasi-geodesic for some $c\in \mathbb{N}$ such that for some $q \in \mathbb{N}$ all vertices of~$G$ have distance at most~$q$ from~$T$, then $c\geq (2/3) \cdot (d - \lceil q/2 \rceil + 1)$.
\end{lemma}

\begin{proof}
	Denote $\lceil q/2 \rceil$ by $p$.
    As otherwise the conclusion is void, we may assume that $d > p - 1$.

    For every sequence $s\neq \varepsilon$ of length at most $d - (p - 1)$ associated with a wheel, the tree $T$ contains $v^s_0$ or~$v^s_1$.
    Indeed, let $t$ be a sequence of integers $4$ of length $p - 1$.
    Since $n \geq 6$, $v^{s}_3$ and $v^{s}_4$ have distance $2$ from~$v^s_0$ and $v^s_1$ in $G$.
    Applying this fact iteratively yields that $v^{st}_3$ has distance $2p\geq q$ from $v^s_0$ and $v^s_1$ in $G$ and even larger distance from all other vertices of~$G^s_1$.
	So $T$ contains a vertex of~$G^s_2$.
    In particular, $T$ contains a vertex in $G^{s'}_2 \subseteq G^s_1$ by applying the above argument on $s'=i$ for some $i$ distinct to the first integer in the sequence of $s$. 
	As $T$ is connected, this implies that $T$ contains $v^s_0$ or~$v^s_1$.

    If it exists, let $s\neq \varepsilon$ be a sequence of length at most $d - (p -1) - (c + 1)$ associated with a wheel and let $t$ be a sequence of integers $2$ of length $c + 1$.
	Assume for a contradiction that $T$ does not contain $v^s_1$.
	Since $v^s_1 = v^{s2}_1 = \cdots = v^{st}_1$ and the length of $st$ is at most $d - (p -1)$, it follows from the above that $T$ contains both $v^s_0$ and $v^{st}_0$. 
    And the path in~$T$ from~$v^s_0$ to~$v^{st}_0$ contains all $v^{st'}_0$ where $t'$ is an initial segment of $t$.
	As the $v^{st'}_0$ are non-adjacent, the distance of $v^s_0$ and $v^{st}_0$ in~$T$ is at least twice the length of~$t$ which in turn is larger than $c$.
	This is a contradiction to the fact that $v^s_0$ and $v^{st}_0$ have distance $2$ in $G$ and that $T$ is $c$-quasi-geodesic.
	So $T$ contains $v^s_1$, and by letting $t$ be a sequence of integers $n$ of length $c+1$ the same argument yields that $T$ contains $v^s_0$.
    Hence, $T$ contains all vertices on rims of wheels of $G$ associated to sequences of length at most $d - (p -1) - (c +1)$.
	So \cref{lem:NocGeodesicSpanningTree} yields $c\geq 2(d - (p -1) - (c + 1) +1)$ and thus $c\geq (2/3) \cdot (d - \lceil q/2 \rceil + 1)$.

    So we are left with the case that there is no sequence $s\neq \varepsilon$ of length at most $d - (p - 1) - (c + 1)$ associated with a wheel, that is, $d - p - c \leq 0$.
    This is equivalent to $c \geq d - p$, so either $d - p \geq (2/3)(d - p + 1)$ and the lemma holds, or $d - p < (2/3) (d - p + 1)$.
    In the latter case $d - p < 2$, so $d - p + 1 \leq 2$.
    Hence in this remaining case either $c \geq 4/3$ and the lemma holds, or $c = 1$ and the lemma holds by \cref{lem:NocGeodesicSpanningTree}.
\end{proof}
}{}

\bibliographystyle{amsplain}
\bibliography{collective_1.bib,local_1.bib}

\end{document}